\documentclass[a4paper,12pt,reqno]{amsart}
\usepackage{float}
\usepackage{graphicx}
\usepackage{tikz}
\usepackage{amsmath,amsthm,amssymb,amsfonts}
\usepackage{caption}
\usepackage{subcaption}
\usepackage{latexsym}
\usepackage{color}     
\usepackage{hyperref}
\usepackage{enumerate}
\usepackage{enumitem}

\usepackage{geometry}
\usepackage[normalem]{ulem}

\newcommand{\RRR}{\color{black}}
\newcommand{\EEE}{\color{black}}

\newcommand{\RR}{\mathbb R}
\def\Eins{{\mathchoice {\rm 1\mskip-4mu l} {\rm 1\mskip-4mu l}%
					{\rm 1\mskip-4.5mu l} {\rm 1\mskip-5mu l}}}
\newcommand{\cY}{\mathcal{Y}}
\newcommand{\cof}{\operatorname{cof}}
\newcommand{\abs}[1]{\left\vert#1\right\vert}
\newcommand{\norm}[1]{\left\|#1\right\|}
\newcommand{\mysetr}[2] {\left\{#1\,\left|\,#2\right.\right\}}
\newcommand{\idmatrix}{\Eins}
\newcommand{\bald}{\begin{aligned}}
\newcommand{\eald}{\end{aligned}}
\newcommand{\Om}{\Omega}

\newcommand{\vect}[1]{\left(\begin{array}{c} #1 \end{array}\right)}

\newtheorem{theorem}{Theorem}

\newtheorem{thm}[theorem]{Theorem}

\newtheorem{lem}[theorem]{Lemma}
\newtheorem{prop}[theorem]{Proposition}
\theoremstyle{definition}

\theoremstyle{remark}
\newtheorem{rem}[theorem]{Remark}

\allowdisplaybreaks

\numberwithin{equation}{section}

\begin{document}

\title[A new Lavrentiev example]{A new example for the Lavrentiev phenomenon\\ in Nonlinear Elasticity}
\author[S. Almi]{Stefano Almi}
\address[Stefano Almi]{University of Naples Federico II, Department of Mathematics and Applications R. Caccioppoli, via Cintia, Monte S. Angelo, 80126 Naples, Italy. }
\email{stefano.almi@unina.it}

\author[S. Kr\"omer]{Stefan Kr\"omer}
\address[Stefan Kr\"omer]{The Czech Academy of Sciences, Institute of Information Theory and Automation, Pod vod\'{a}renskou v\v{e}\v{z}\'{\i}~4, 182~08~Praha~8, Czech Republic. }
\email{skroemer@utia.cas.cz}

\author[A. Molchanova]{Anastasia Molchanova}
\address[Anastasia Molchanova]{University of Vienna, Faculty of Mathematics, Oskar-Morgenstern Platz 1, 1090 Vienna, Austria. }
\email{anastasia.molchanova@univie.ac.at}

\date{\today}

\begin{abstract}
	We present a new 
	example for the Lavrentiev phenomenon in context of nonlinear elasticity, caused by an interplay of the elastic energy's resistance to infinite compression
	and the Ciarlet--Ne\v{c}as condition, a constraint preventing global interpenetration of matter  
	on sets of full measure.
\end{abstract}


\maketitle

\let\thefootnote\relax\footnotetext{Corresponging author: Stefano Almi, University of Naples Federico II, Department of Mathematics and Applications R. Caccioppoli, via Cintia, Monte S. Angelo, 80126 Naples, Italy.\\ Email: stefano.almi@unina.it}


\section{Introduction\label{sec:intro}}

Following the by-now classical theory of nonlinear elasticity \cite{Ant05B,Ba77a,Cia88B,Sil97B}, we consider an elastic body occupying in its reference configuration an open bounded set~$\Om \subseteq \RR^{d}$ 
with Lipschitz boundary~$\partial \Om$, subject to a prescribed boundary condition on a part $\Gamma\subset \partial\Om$
with positive surface measure, i.e., $\mathcal{H}^{d-1}(\Gamma)>0$. 
A possible deformation of the body is described by a mapping
$y\colon \Omega \to \RR^{d}$ 
such that $y = y_0$ on $\Gamma$, where $y_0$ is the imposed boundary data.
Its associated internally stored elastic energy is given by the functional 
\begin{equation}\label{def:energy}
	E(y):=\int_{\Omega} W(\nabla y(x)) \, dx,
\end{equation}
with a function $W$ representing material properties: the local energy density,  
which \RRR is here assumed to be only a function of the deformation gradient and not of the position~$x$. \EEE
A crucial aspect of this mathematical model~\cite{Ball2010} is to define a suitable class of admissible deformations that  capture relevant features, such as non-interpenetration of matter, which mathematically translates into injectivity of~$y$.
However, considering different admissible classes can lead to a Lavrentiev phenomenon, i.e., the functional infima differ when restricting the minimization of ~\eqref{def:energy} to more regular deformations, such as~$W^{1, \infty}$ in place of~$W^{1, p}$.
Functionals demonstrating this behavior were first discovered in the early 20th century \cite{La1927a,Man1934}. 
There the minimum value over $W^{1,1}$ is strictly less than the infimum over~$W^{1,\infty}$.  
For an extensive survey on the Lavrentiev phenomenon in a broader context, we refer the interested reader to \cite{ButBel1995}.

In the context of nonlinear elasticity, the Lavrentiev phenomenon 
was first observed with admissible deformations that allow cavitations, i.e., the formation of voids in the material \cite{Bal1982}.
For the study of cavitations, we refer to~\cite{BarDou2020,HeMo15a} and references therein.

A natural question raised in \cite{BalMiz1985} and \cite{Ba02a} is: \\
\textit{Can the Lavrentiev phenomenon occur for elastostatics under growth conditions on the stored-energy function, ensuring that all finite-energy deformations are continuous?}\\
  This is indeed the case, and the first example of this kind 
has been given in two dimensions \cite{Foss2003,FosHruMiz2003b,FoHruMi03a}.
It features an energy density with desirable properties: $W$ is smooth, polyconvex, frame-indifferent, isotropic, 
$W(F) \gtrsim |F|^{p}$ with $p>2$, and $W(F) \to \infty$ as $\det F \to 0+$. Moreover, admissible deformations are almost everywhere (a.e.) injective.
In these examples, the reference configuration is represented by a disk sector $\Om_{\alpha}: = \{ r (\cos \theta,\sin \theta): 0<r<1, 0<\theta < \alpha\}$. 
A crucial aspect for the emergence of the Lavrentiev phenomenon in that example is the local behavior of (almost) minimizers near the tip at $r=0$, interacting with a particular choice of boundary conditions.
The latter fix   
the origin $y(0,0) = (0,0)$, $y(1,\theta) = (1, \frac{\beta}{\alpha} \theta)$ and $y(\Om_\alpha) \subset \Om_\beta$, where $0 < \beta < \frac{3}{4} \alpha$. 

In the current paper, we provide examples of the Lavrentiev phenomenon in elasticity both in two and three dimensions. 
The elastic energy is of a simple neo-Hookean form with physically reasonable properties as described above, and admissible deformations are continuous and a.e.-injective. 
Differently from~\cite{Foss2003,FosHruMiz2003b,FoHruMi03a}, the Lavrentiev phenomenon in our example
is not related to the local behavior of almost minimizers near prescribed boundary data, but to a possible global self-intersection of the material that still maintains a.e.~injectivity by compressing two different material cross-sections to a single point (or line in 3D) of self-contact in deformed configuration. It turns out to be energetically favorable due to our particular choice of boundary conditions but is no longer possible if we restrict to a sufficiently smooth class of admissible deformations. This then leads to a higher energy infimum. 

Throughout the paper, we consider \emph{locally orientation preserving} deformations with $p$-Sobolev regularity
\begin{equation}\label{def:W1p+}
	W_+^{1,p} (\Om \RRR;\EEE \RR^{d}) := \{ y \in W^{1,p} (\Om; \RR^{d})\mid \det\nabla y>0~\text{a.e.~in $\Om$}\}\subset W^{1,p}(\Om; \RR^{d}).
\end{equation}
If $p>d$, the Sobolev embedding theorems ensure the continuity of $W^{1,p}$-mappings.
The question of injectivity of deformations, i.e., non-interpenetration of matter, is more delicate and it has been extensively studied. 
 Let us mention just a few references.
For local invertibility conditions, see~\cite{BarHenMor2017, FoGa95a, Hen-Str-2020}.   
As for global injectivity one may ask some 
coercivity with respect to specific ratios of powers of a matrix $F$, its cofactor matrix $\cof F$, and its determinant $\det F$
  combined with global topological information from boundary values \cite{Ba81a,HeMoOl21a,IwaOnn2009,Kroe20a,MolVod2020,Sve88a} or 
second gradient~\cite{HeaKroe09a},
as well as other regularity~\cite{CiaNe87a,Tan1988,Sve88a}
and topological restrictions such as (INV)-condition~\cite{BarHenMor2017,ConDeLel2003,HenMor2010,MulSpe1995,SwaZie2004}
and considering limits of homeomorphisms~\cite{BouHenMol2020,DolHenMol2022,IwaOnn2009,MolVod2020}.
In this paper, we adopt the approach from~\cite{CiaNe87a}, where the authors investigate a class of mappings $y\in W_+^{1,p} (\Om; \RR^{d})$ satisfying the \textit{Ciarlet--Ne\v{c}as condition}:
\begin{align}\label{CN}\tag{CN}
	\int_\Omega \det(\nabla y (x)) \,dx \leq \abs{y(\Omega)},
\end{align} 
and prove that the mappings of this class are a.e.-injective.

In the examples we consider $W(F) \gtrsim |F|^{p} + (\det F)^{-q}$,
the reference configuration $\Omega$ \RRR in dimension $d=2,3$. \EEE The boundary data $y_0$ are chosen in such a way that the energy $E$ favors deformations that have nonempty sets of non-injectivity. In particular, we construct in Section~\ref{sec:Lav2d} 
(resp.~Section~\ref{sec:Lav3d}) a competitor $y \in W^{1, p}_{+} (\Om; \RR^{2})$ (resp.~$y \in W^{1, p}_{+}(\Om; \RR^{3})$) satisfying the Ciarlet--Ne\v{c}as condition~\eqref{CN} and having a line (resp.~a plane) of non-injectivity. 
The energy of such deformation is shown to be strictly less than that of
Lipschitz deformations, for which injectivity is ensured everywhere.   
The global injectivity in this case follows from the Reshetnyak theorem for mappings of finite distortion \cite{ManVill1998}.
Specifically, a mapping $y\in W^{1,d}_{loc}( \Om; \RR^{d})$ with $\det \nabla y \geq 0$ a.e.\ has finite distortion if $|\nabla y (x)| = 0$ whenever $\det \nabla y (x) = 0$. 
If, in addition, the distortion $K_{y}:=\frac{|\nabla y |^d}{\det \nabla y} \in L^{\varkappa}$ with $\varkappa > d-1$, then 
$y$ is either constant or open and discrete.
Furthermore, it is not difficult to see that an a.e.-injective and open mapping $y\in W^{1,d}_{loc} (\Om; \RR^{d})$ is necessarily injective everywhere, as pointed out in
\cite[Lemma 3.3]{GraKruMaiSte19a}. 
For a general theory of mappings of finite distortion the reader is referred to 
\cite{HenKos2014}.

Our example also shows that, depending on the precise properties of the energy density~$W$, 
there can be an energy gap between the class of orientation preserving a.e.~injective deformations (i.e., satisfying the Ciarlet--Ne\v{c}as condition) on the one hand and the strong (or weak) closure of Sobolev homeomorphisms in the ambient Sobolev space on the other hand.
If these classes do not coincide (which can certainly happen if there is not enough control of the distortion via the energy to apply the Reshetnyak theorem \cite{ManVill1998} as above, \RRR see \cite[Fig.~4]{Pa08a}, e.g.), \EEE one has to carefully choose which constraint to use to enforce non-interpenetration of matter, even if $p>d$. In our example,
the Ciarlet--Ne\v{c}as condition does allow a ``deep'' self-interpenetration in such a scenario. As a matter of fact, this self-interpenetration is also topologically stable in the sense that all $C^0$-close deformations still self-intersect (see Figures~\ref{f:reference_configuration} and 
\ref{f:cross} for reference and deformed configurations in the 2D case). To us, it seems doubtful that such a deformation corresponds to a physically meaningful state. This strongly speaks for preferring a closure of homeomorphisms as the admissible class in such cases. 
An open problem in this context is to find sharp conditions for the energy density 
so that all a.e.-injective orientation preserving Sobolev maps can be found as strong (or weak) limits of Sobolev homeomorphisms in $W^{1,p}$.
In case $p\geq d$, having $K_{y}\in L^{\varkappa}$ with $\varkappa > d-1$ as above is clearly sufficient, but probably not necessary, at least not in dimension $d\geq 3$.

The plan of the paper is the following.
Section \ref{sec:ex} is dedicated to the general setting of the problem and a few basic auxiliary results.   
In Sections \ref{sec:Lav2d} and \ref{sec:Lav3d}, we discuss the Lavrentiev phenomenon in dimensions two and three for the energy $E$ in the class of deformations
$y\in W_+^{1,p} (\Om, \RR^{d})$ satisfying the Ciarlet--Ne\v{c}as condition \eqref{CN} as well as suitable Dirichlet boundary conditions on selected parts of $\partial \Omega$.

\section{General setting\label{sec:ex}}

In dimension $d \geq 2$ we consider a Neo-Hookean nonlinear elastic material with energy density:
\begin{align}\label{W0}
  W(F):=\begin{cases}
		\abs{F}^p+\gamma \dfrac{1}{(\det F)^q} & \text{if $\det F>0$},\\
		+\infty & \text{else,}
		\end{cases}
		\qquad \text{for $F\in \RR^{d\times d}$.}
\end{align}
In \eqref{W0}, $\abs{F}:=\big(\sum_{ij} F^2_{ij}\big)^{\frac{1}{2}}$ denotes the standard Euclidean matrix norm,
\RRR $p>d$ and $q>0$ are constants, \EEE and $\gamma>0$ is chosen in such a way that~$W$ is minimized at the identity matrix~$\idmatrix$,
i.e.,
\begin{equation}\label{def:gamma}
	\gamma := \frac{p d^{\frac{p}{2}-1}}{q}.
\end{equation}
Indeed, let $\det F >0$ and $\lambda_1, \ldots, \lambda_d >0$ be singular values of $F$, then 
\begin{displaymath}
	W(F)=\mathcal{W}(\lambda_1,\dots,\lambda_d) = \Big(\sum_{k=1}^{d} \lambda_k^2\Big)^{\frac{p}{2}} + \gamma \Big(\prod_{k=1}^{d}\lambda_k\Big)^{-q} 
\end{displaymath}
and equalities 
$\frac{\partial}{\partial \lambda_i}\mathcal{W}\mid_{\lambda_j=1}=0$ 
give us \eqref{def:gamma}.
Moreover, 
$\lambda_i =1$, $i=1,\ldots d$, 
is the global minimizer of $\mathcal{W}$.
Indeed, if 
$(\lambda_1, \ldots \lambda_d)$ 
is a local minimum, then for any 
$i = 1,\ldots d$,
\begin{displaymath}
	\frac{\partial}{\partial \lambda_i}\mathcal{W}
	=p\lambda_i \Big(\sum_{k=1}^{d} \lambda_k^2\Big)^{\frac{p}{2}-1}  - \gamma q \frac{1}{\lambda_i} \Big(\prod_{k=1}^{d}\lambda_k\Big)^{-q}
	=0.
\end{displaymath}
Therefore, \RRR $\lambda_i=1$ for all $i= 1,\ldots d$. \EEE
In other words only rotation matrices $F \in SO(d)$ are minimizers of $W$.

Below we summarize some ``good'' \cite{Ball2010} properties of the energy density $W$.
\begin{prop}
\label{p:properties-W}
	For $W$ given by \eqref{W0} and \eqref{def:gamma},
	we have that 
	\begin{enumerate}
		\item $W \in C^{\infty}(\RR^{d\times d}_{+}\RRR ;\EEE \RR)$,
		\item $W(F) \to +\infty$ as $\det F \to 0{+}$,
		\item $W$ is frame-indifferent and isotropic, i.e.\ $W(RF) = W(FR) = W(F)$ for all $R\in SO(d)$ and $F\in \RR^{d\times d}_{+}$,
		\item $W$ is polyconvex,
		\item $W(F) - W(\idmatrix) \geq 0$ for all $F\in \RR^{d\times d}_{+}$,
		\item there exist constants $c=c(d,p,q)>0$ and $b=b(d,p,q)\in \mathbb{R}$ such that 
		\begin{align}\label{Wcoercivity}
			W(F) \geq c \left(\abs{F}^p + \abs{\cof F}^{\frac{p}{d-1}} + (\det F)^{\frac{p}{d}}\right) +b.
		\end{align}
	\end{enumerate}
\end{prop}

For later use, we point out the following proposition, which express the minimality of the identity map in a quantitative form. \RRR From now on, \EEE $\| F\|_{2}$ \RRR denotes \EEE the operator norm of $F \in \mathbb{R}^{d \times d}$, i.e., $\norm{F}_2:=\sup\{ \abs{Fe}:\abs{e}=1\}$.

\begin{prop}
\label{p:lower-bound}
	For $W$ given by \eqref{W0} and \eqref{def:gamma},
	we have the following lower bound
	\begin{align}\label{Wlowerbound}
		W(F)-W(\idmatrix) \geq c \abs{\norm{F}_2-1}^p+c(\norm{F}_2-1)^2
	\end{align}
	\RRR for \EEE some constant $c=c(d,p,q)>0$.
\end{prop}

\begin{proof}
As before, we express $W(F)=\mathcal{W}(\lambda_1,\ldots,\lambda_d)$ and
$\norm{F}_2=\max\{\lambda_i:i=1,\ldots,d\}$ in terms of the singular values of $F$.
Abbreviating
\begin{displaymath}
	\textstyle S:=\Big(\sum_{k=1}^{d} \lambda_k^2\Big)^{\frac{1}{2}},\quad P:=\prod_{k=1}^{d}\lambda_k,
\end{displaymath}
we have that $\mathcal{W}(\lambda_1,\ldots,\lambda_d) = S^p+\gamma P^{-q}$. Thus, it holds
\begin{displaymath}
	\begin{aligned}
		&\frac{\partial^2}{\partial \lambda_i \partial \lambda_j} \mathcal{W}(\lambda_1,\dots,\lambda_d) 
		= \left(p(p-2)\lambda_i\lambda_j+p\delta_{ij} S^{2}\right)
		S^{p-4}
		+\gamma q(q+\delta_{ij}) \frac{1}{\lambda_i\lambda_j} P^{-q}.
	\end{aligned}
\end{displaymath}
Notice that
$\big(p(p-2)S^{p-4}\lambda_i\lambda_j\big)_{ij}$ 
and 
$\big(\gamma q^2 P^{-q} \lambda_i^{-1}\lambda_j^{-1}\big)_{ij}$ 
are positive semidefinite matrices of rank $1$, 
while the other contributions involving Kronecker's 
$\delta_{ij}$ 
give a diagonal matrix with positive coefficients that can be estimated for all 
$i=j$. 
Indeed, defining
\begin{displaymath}
	\mu=\mu(d,q):=\min \left\{
	\lambda_j^{-2} P^{-q} \mid 1 \geq \lambda_1,\ldots,\lambda_d>0,~S=1
	\right\}\geq d^{\frac{qd}{2}}>0
\end{displaymath}
(due to symmetry, $\mu$ does not depend on $j$), we obtain that
\begin{align}
\label{e:SP}
	& p S^{p-2}+\gamma q \lambda_j^{-2}P^{-q}\geq p S^{p-2}+\gamma q \mu S^{-qd-2}.
\end{align}
Choosing 
$\alpha \in (0, 1)$ 
such that 
$p \leq \frac{2- \alpha}{1 - \alpha}$ 
we may continue in~\eqref{e:SP} with
\begin{align*}
	& p S^{p-2}+\gamma q \lambda_j^{-2}P^{-q}\geq (1 - \alpha) p(p-1) S^{p-2}+ \big( (\alpha - 1) p^{2} + (2 - \alpha) p\big) S^{p-2} + \gamma q \mu S^{-qd-2},
\end{align*}
from which we infer the existence of a constant $\hat{c}=\hat{c}(\gamma,d,p,q,\mu)>0$ such that
\begin{displaymath}
	p S^{p-2}+\gamma q \lambda_j^{-2}P^{-q}\geq \hat{c} (p(p-1)S^{p-2}+2).
\end{displaymath}
Altogether, we get 
\begin{align}\label{D2Wc-posdef}
	\xi\cdot D^2\mathcal{W}(\lambda_1,\ldots,\lambda_d)\xi\geq \hat{c} (p(p-1)S^{p-2}+2)\abs{\xi}^2\quad\text{for $\xi\in \RR^d$}.
\end{align}

We now conclude for~\eqref{Wlowerbound}. Without loss of generality, we may assume that $\norm{F}_2=\lambda_1$. Let us define the curve
$\lambda\colon[0,1]\to (0, +\infty)^d$ connecting~$(\lambda_{1}, \ldots, \lambda_{d})$ to $(1, \ldots, 1)$
\begin{displaymath}
	t\mapsto \lambda(t)=(\lambda_1(t),\ldots,\lambda_d(t)):=(t\lambda_1+1-t,\ldots,t\lambda_d+1-t).
\end{displaymath}
Since $\frac{\partial}{\partial \lambda_j}\mathcal{W}(1,\ldots,1)=0$, integrating  twice along~$\lambda$ and using~\eqref{D2Wc-posdef} we obtain
\begin{displaymath}
	\begin{aligned}
		W(F)-W(\idmatrix)&=\int_0^1 \int_0^\tau \frac{d^2}{dt^2}\mathcal{W}\big(\lambda_1(t),\ldots,\lambda_d(t)\big)\,dt\,d\tau \\
		&\geq \hat{c} \int_0^1 \int_0^\tau (p(p-1)\lambda_1(t)^{p-2}+2)\dot\lambda_1(t)^2\,dt\,d\tau \\
		&= \hat{c} \int_0^1 \int_0^\tau  \frac{d}{dt} \Big[\big(p\lambda_1(t)^{p-1}+2\lambda_1(t)\Big] \dot\lambda_1(t)\,dt\,d\tau \\
		& = \hat{c} \left( \lambda_1^{p}-1-p(\lambda_1-1)+(\lambda_1-1)^2 \right) \geq c \left(|\lambda_1-1|^{p}+(\lambda_1-1)^2\right),
	\end{aligned}
\end{displaymath}
with $c:=\min\big\{\hat{c},\frac{1}{2}\big\}$ (exploiting that $p\geq 2$).
\end{proof}

In both the examples we present in this paper, we fix as reference configuration an open bounded set~$\Om_{s} \subseteq \RR^{d}$ with Lipschitz boundary~$\partial \Om_{s}$. Our set will always have two connected components whose precise shape will be chosen depending on the dimension $d$ and will further depend on a parameter $s>0$. 
For every $y \in W_+^{1,p} (\Om_{s}, \RR^{d})$ we define the energy functional
\begin{align*}\label{Eelastic}
	E_s(y):=\int_{\Omega_s} \left(W(\nabla y)-W(\idmatrix)\right) \,dx.
\end{align*}
In particular, notice that the energy $E_{s}$ is normalized to $0$ at $y = id$, since $W$ attains minimum value on $SO(d)$. 
Let~$\Gamma_{s}$ be a subset of $\partial\Om_{s}$, $\mathcal{H}^{n-1}(\Gamma_s)>0$, with imposed Dirichlet boundary data $y_{0}\in W^{1, p}_{+} (\Om_{s}, \RR^{d})$.
The set of admissible deformations~$\cY_s \subseteq W^{1, p}_{+} (\Om_{s}, \RR^{d})$ reads as
\begin{equation}\label{def:Ys}
	\cY_s:=\mysetr{y\in W_+^{1,p}(\Omega_s;\RR^d)}{\text{\eqref{CN} holds and $y=y_0$ on $\Gamma_s$}}.
\end{equation}

The existence of minimizers is nowadays classic and follows,  e.g., from~\cite[Theorem 5]{CiaNe87a} due to Proposition~\ref{p:properties-W} since $p>d$.

\begin{thm}
	If $\mathcal{Y}_{s}\neq \varnothing$ and 
	$\inf\limits_{{y} \in \mathcal{Y}_s} E_s(y)<\infty$,
	then there exists $\hat{y}_s \in \mathcal{Y}_{s}$ such that 
	$\inf\limits_{{y} \in \mathcal{Y}_s} E_s(y) = E_s(\hat{y}_s)$.
\end{thm}


\section{The Lavrentiev phenomenon in dimension two\label{sec:Lav2d}}

In dimension $d=2$ we consider a reference configuration~$\Om_{s}$ consisting of two stripes of width $0<s<1$,
given by (see also Fig.~\ref{f:reference_configuration})
\begin{equation}\label{def:omega_s}
	\bald
		&\Omega_s:=S_1 \cup S_2,\quad\text{where}~~S_1:=(-1,1)\times (-s,s),~~S_2:=\xi+QS_1, \\
		&~~Q:=\left(\begin{array}{rr}0 & -1\\1 & 0\end{array}\right)~~\text{and}~~\xi:=\vect{4\\0}.
	\eald
\end{equation}

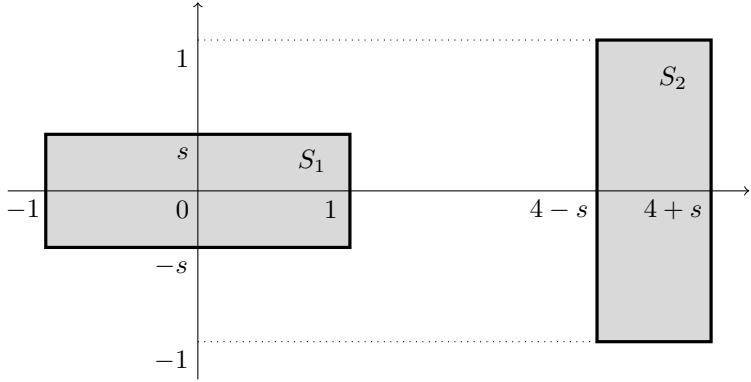
\begin{figure}[h!]
	\begin{tikzpicture}[every node/.style={font=\footnotesize}]
		\filldraw[fill= gray, very  thick, fill opacity=0.3] (-2,-0.75) rectangle (2,0.75);
		\filldraw[fill= gray, very  thick, fill opacity=0.3] (5.25,-2) rectangle (6.75,2);
		\node at (-0.2, -0.25) {$0$};
		\node at (1.5, 0.4) {$S_{1}$};
		\node at (6.25, 1.5) {$S_{2}$};
		\node at (-0.35, -1) {$-s$};
		\node at (-0.2, 0.5) {$s$};
		\node at (-2.3, -0.25) {$-1$};
		\node at (1.75, -0.25) {$1$};
		\node at (-0.35, -2.25) {$-1$};
		\node at (-0.2, 1.75) {$1$};
		\node at (4.75, -0.25) {$4-s$};
		\node at (6.25, -0.25) {$4+s$};
		\draw[black, ->] (-2.5,0) -- (7.25, 0);
		\draw[black, ->] (0,-2.5) -- (0, 2.5);
		\draw[black, dotted] (0,2) -- (5.25, 2);
		\draw[black, dotted] (0,-2) -- (5.25, -2);
	\end{tikzpicture}
	\caption{Visualization of the reference configuration~$\Om_{s} = S_{1} \cup S_{2}$.} 
	\label{f:reference_configuration}
\end{figure}

We denote by~$\Gamma_{s}$ the subset of~$\partial\Om_{s}$ given by 
\begin{displaymath}
	\Gamma_s:= \big[ \{-1,1\}\times (-s,s) \big] \cup \big[ (4-s,4+s)\times \{-1,1\} \big] \subset \partial\Omega_s.
\end{displaymath}
On $\Gamma_{s}$ we impose the following Dirichlet boundary condition
\begin{displaymath}
	y_0(x):=
	\begin{cases}
		x & \text{for $x\in \overline{S_1}$,}\\
		x-\xi &\text{for $x\in \overline{S_2}$}
	\end{cases}
\end{displaymath}
Notice that on both pieces of $\Omega_s$ the function $y_{0}$ is such that~$\nabla y_0=\idmatrix$, which minimizes $W$ pointwise. However,~$y_0(\Omega_s)$ is cross-shaped with $y_0$ doubly-covering the center. Hence,~$y_0$ is not globally injective and does not satisfy \eqref{CN}. It is not hard to see that $\cY_s$, defined by \eqref{def:Ys}, still contains many admissible functions as long as $s<1$.

\RRR
\begin{rem}
Our reference configuration $\Omega_s$ is not connected, but this is not essential for our examples, just convenient. In fact, we could 
add a connecting piece $S_3$ to $\Omega_s$, say from  $\{-1\}\times (-s,s)$ (the left edge of $S_1$) to $(4-s,4+s)\times \{1\}$ (the upper edge of $S_2$), while still imposing the Dirichlet condition on $\Gamma_s$ as before. This would not affect our analysis near the possible self-intersection which only involves $S_1\cup S_2$ (cf.~Fig.~\ref{f:cross} and Fig.~\ref{f:around} below), but it would create extra technical hassle, as we would then have to control behavior of $y$ and the minimal energy contribution on $S_3$ as well. A more refined example for the extended reference configuration $\tilde\Omega_s:=S_1\cup S_2\cup S_3$ could even try to replace the ``inner'' part of the Dirichlet condition on $(\{-1\}\times (-s,s))\cup ((4-s,4+s)\times \{1\})\subset \Gamma\cap \tilde\Omega_s$ by an obstacle that the deformations are forced to wrap around. The precise formulation and analysis for this would be much more challenging and technical, though.
\end{rem}
\EEE

\begin{thm}[The Lavrentiev phenomenon occurs]\label{thm:lav2d}
	Let $p \in (2, +\infty)$ and $q \in (1, \frac{p}{p-2}]$. Then, there exists $\overline{s} \in (0, 1)$ such that for every $s \in (0, \overline{s}]$ the following holds: 
	\begin{equation}
	\label{e:lav2d}
		\inf_{y\in W^{1,\infty} (\Om_{s}; \RR^{d}) \cap \cY_s}E_s(y) >
		\min_{y\in \cY_s}E_s(y).
	\end{equation}
\end{thm}

The proof of Theorem~\ref{thm:lav2d} is a consequence of the following two propositions, which determine the asymptotic behavior of the minimum problems in~\eqref{e:lav2d}.

\begin{prop}
\label{p:lav-2d-1}
	Let $p \in (2, +\infty)$ and $q \in (1, \frac{p}{p-2}]$. Then 
	$\min_{y\in \cY_s} E_s(y)=o(s)$, i.e.,
	we have that
	\begin{equation}
	\label{e:step1}
		\lim_{s\searrow 0} \frac{1}{s}\inf_{y \in \cY_{s}} \,  E_{s}(y) =0.
	\end{equation}
\end{prop}

\begin{prop}
\label{p:lav-2d-2}
	Let $p \in (2, +\infty)$ and $q \in (1, \frac{p}{p-2}]$. Then, there exists $\overline{s} \in (0, 1)$ such that for every $s \in (0, \overline{s}]$,
	\begin{equation}	
		\label{e:step2}
		ms\leq \inf_{y \in W^{1, \infty}(\Omega_{s}; \RR^{2}) \cap \cY_{s}} \, E_{s}(y) \leq Ms.
	\end{equation}
	with constants $0<m<M <+\infty$ independent of $s$.
\end{prop}

We start with the proof of Proposition~\ref{p:lav-2d-1}.

\begin{proof}[Proof of Proposition~\ref{p:lav-2d-1}] 
In order to prove~\eqref{e:step1} we explicitly construct a deformation~$y_{\alpha,\beta}$ forming a cross with self-intersection. By squeezing with suitable rate two central cross-sections to a point, which will be the only point of intersection in $y_{\alpha,\beta}(\Omega)$, we produce an almost-minimizer of~$E_{s}$ in $\cY_s$. 

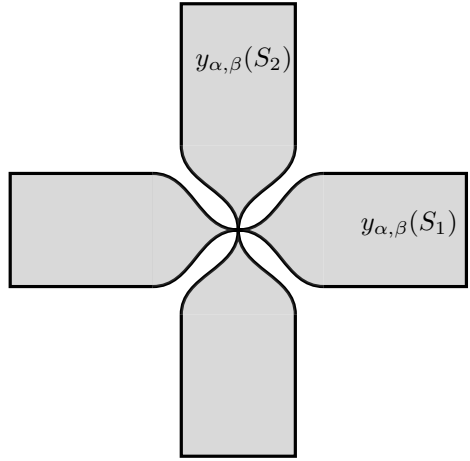
\begin{figure}[h!]
	\begin{tikzpicture}[scale = 1.5, every node/.style={font=\footnotesize}]
		\fill[fill= gray, fill opacity=0.3] (-2,-0.5) rectangle (-0.75,0.5);
		\fill[fill= gray, fill opacity=0.3] (0.75,0.5) rectangle (2,-0.5);
		\draw[very thick] (-0.75,0.5) -- (-2,0.5) -- (-2,-0.5) -- (-0.75,-0.5);
		\draw[very thick] (0.75,0.5) -- (2,0.5) -- (2,-0.5) -- (0.75,-0.5);

		\fill[fill= gray, fill opacity=0.3] (-0.5,-2) rectangle (0.5,-0.75);
		\fill[fill= gray, fill opacity=0.3] (-0.5,2) rectangle (0.5,0.75);
		\draw[very thick] (-0.5,-0.75) -- (-0.5,-2) -- (0.5,-2) -- (0.5,-0.75);
		\draw[very thick] (-0.5,0.75) -- (-0.5,2) -- (0.5,2) -- (0.5,0.75);

		\draw[very thick, smooth] (-0.75, -0.5) to [in=180,out=0] (0,0) to [in=0,out=180] (-0.75, 0.5);
		\fill[fill= gray, fill opacity=0.3] (-0.75, -0.5) to [in=180,out=0] (0,0) to [in=0,out=180] (-0.75, 0.5);

		\draw[very thick, smooth] (0.75, -0.5) to [in=0,out=180] (0,0) to [in=180,out=0] (0.75, 0.5);
		\fill[fill= gray, fill opacity=0.3] (0.75, -0.5) to [in=0,out=180] (0,0) to [in=180,out=0] (0.75, 0.5);

		\draw[very thick, smooth] (-0.5, -0.75) to [in=-90,out=90] (0,0) to [in=90,out=-90] (0.5, -0.75);
		\fill[fill= gray, fill opacity=0.3] (-0.5, -0.75) to [in=-90,out=90] (0,0) to [in=90,out=-90] (0.5, -0.75);

		\draw[very thick, smooth] (-0.5, 0.75) to [in=90,out=-90] (0,0) to [in=-90,out=90] (0.5, 0.75);
		\fill[fill= gray, fill opacity=0.3] (-0.5, 0.75) to [in=90,out=-90] (0,0) to [in=-90,out=90] (0.5, 0.75);

		\node at (1.5, 0.05) {$y_{\alpha, \beta}(S_{1})$};
		\node at (0.05, 1.5) {$y_{\alpha, \beta}(S_{2})$};
	\end{tikzpicture}
	\caption{$y_{\alpha, \beta}(S_{1})$ and $y_{\alpha, \beta}(S_2)$.} \label{f:cross}
\end{figure}

We start with the case $q \in (1, \frac{p}{p-2})$. We \RRR divide \EEE $S_1$ into two \RRR subsets \EEE $S'_1=S_1\cap\{|x_1|\leq s\}$ and $S''_1=S_1\cap\{|x_1|\geq s\}$ and fix $\frac{p-1}{p}<\alpha<\beta\leq1$. For $x\in S'_1$ we set
\begin{displaymath}
	y_{\alpha,\beta}(x):=\vect{\displaystyle \frac{x_1}{\abs{x_1}^{1-\alpha}} \\[5mm] \displaystyle  \frac{\abs{x_1}^\beta}{s^\beta} x_2},\quad\text{whence}\quad \nabla y_{\alpha,\beta}(x)=
	\left(\begin{array}{cc}
		\displaystyle \frac{\alpha}{\abs{x_1}^{1-\alpha}}  & 0	\\[5mm] 
		\displaystyle \frac{\beta}{s^\beta} {\abs{x_1}^{\beta-2} x_1 x_2}  & \displaystyle\frac{\abs{x_1}^\beta}{s^\beta}
	\end{array}\right).
\end{displaymath}
For $x\in S''_1$ we connect $y_{\alpha, \beta}$ to the boundary datum $y_{0}$ as follows:
\begin{displaymath}
	y_{\alpha,\beta}(x):=\vect{\displaystyle \frac{x_{1}}{|x_{1}|} \, \left(\frac{1-s^\alpha}{1-s} ( |x_1| - 1) + 1\right) \\[4mm] x_2},\quad\text{whence}\quad \nabla y_{\alpha,\beta}(x)=
	\left(\begin{array}{cc}
		\displaystyle \frac{1-s^\alpha}{1-s} & 0	\\[4mm] 
		0  & 1
	\end{array}\right).
\end{displaymath}
In particular, $\det\nabla y_{\alpha,\beta}=\frac{1-s^\alpha}{1-s} >0$ in $S''_1$ and $\det\nabla y_{\alpha,\beta}=\frac{\alpha}{s^\beta} \abs{x_1}^{\alpha+\beta-1} >0$ a.e.~in $S'_1$ and
\begin{displaymath}
	\frac{1}{(\det\nabla y_{\alpha,\beta})^q}=
	\begin{cases}
		\displaystyle \frac{s^{\beta q}}{\alpha^q}\abs{x_1}^{(1-\alpha-\beta) q}, &\text{if } x\in S'_1,\\[3mm] 
		\displaystyle\Big( \frac{1-s}{1-s^\alpha}\Big) ^{q}, &\text{if } x\in S''_1,\\
	\end{cases}
\end{displaymath}
Moreover, for $x\in S'_1$ we have that
\begin{displaymath}
	|\nabla y_{\alpha,\beta}| = \left(\frac{\alpha^2}{\abs{x_1}^{2(1-\alpha)}} + \frac{\beta^2}{s^{2\beta}}\abs{x_1}^{2\beta - 2} \abs{x_2}^2 + \frac{\abs{x_1}^{2\beta}}{s^{2\beta}}\right)^{\frac{1}{2}}
	\leq \frac{\alpha}{\abs{x_1}^{1-\alpha}} + \beta \frac{ s^{1 - \beta }}{|x_{1}|^{1 - \beta }}  +1.
\end{displaymath}
Thus, $(\det\nabla y_{\alpha,\beta})^{-q}+|\nabla y_{\alpha,\beta}|^p\in L^1(S_1)$ as long as 
\begin{align}\label{p-q-alpha}
	(1-\alpha -\beta)q>-1, \quad p(\alpha-1)>-1, \quad p(\beta - 1) >-1.   
\end{align}
Such restrictions on $\alpha$ and $\beta$ can be satisfied  whenever $ q \in (1, \frac{p}{p-2})$ by choosing $\alpha\in(\frac{p-1}{p},1)$ and $\beta \in (\alpha,1]$ accordingly.   

We now estimate the behavior of the energy $E_{s} (y_{\alpha, \beta})$ as $s\to 0$. Below, the symbol~$\lesssim$ stands for an inequality up to a positive multiplicative constant independent of $s\in (0,1]$ and $x\in S_1$. We further write~$\approx$ if such inequalities hold in both directions.
By minimality of the identity matrix, by definition of~$W$, and by construction of~$y_{\alpha,\beta}$ on~$S'_{1}$, we have that
\begin{equation}\label{W-upperboundcross}
\begin{aligned}
	0& \leq W(\nabla y_{\alpha,\beta}) - W(\idmatrix)  \lesssim |\nabla y_{\alpha,\beta}|^p+\frac{1}{(\det\nabla y_{\alpha,\beta})^q} 
	\\
	&
	\lesssim \abs{x_1}^{p(\alpha-1)} + \beta  s^{p(1 - \beta)} |x_{1} |^{p( \beta-1) }  + 1 + s^{\beta q} |x_1|^{(1-\alpha-\beta) q} . 
\end{aligned}
\end{equation}
Moreover, since $0<\alpha<1$, the mean value theorem gives
\begin{displaymath}
	\frac{1-s^\alpha}{1-s}-1 \approx s^{\alpha}+ s \approx s^{\alpha}.
\end{displaymath}
This means that on $S_1''$ it holds that $\abs{\nabla y_{\alpha,\beta} -\idmatrix}\lesssim s^{\alpha}$ uniformly in $x$.
By Taylor expansion of $W$ at $\idmatrix$ (where $DW(\idmatrix)=0$ by definition of $\gamma$), we infer that
\begin{align}\label{W-upperboundrest}
	0\leq W(\nabla y_{\alpha,\beta}) - W(\idmatrix)\lesssim s^{2\alpha}\quad\text{on}~S_1''.
\end{align}

Combining \eqref{W-upperboundcross} and \eqref{W-upperboundrest},
we obtain the following upper bound for the energy for all sufficiently small $s$
as long as \eqref{p-q-alpha} holds:
\begin{align}
\label{e:estimate-S1}
	\int_{S_{1}} & W(\nabla y_{\alpha,\beta}) - W(\idmatrix) \, dx 
	\\
	& \lesssim  \int_{S'_1} 
	\left(\abs{x_1}^{p(\alpha-1)}+ s^{p(1 - \beta)} |x_{1} |^{p( \beta-1) } + s^{\beta q} |x_1|^{(1-\alpha-\beta) q} +1\right)\,dx  + \int_{S''_1} s^{2\alpha} \,dx	 \nonumber\\
	& = \frac{4s \cdot s^{p\alpha-p+1}}{1+p(\alpha-1)}+  \frac{4s^{2}}{1+p(\beta-1)}+ \frac{4s \cdot s^{\beta q}  \cdot s^{1+(1-\alpha -\beta) q}}{1+(1-\alpha-\beta) q}+4s^2+4s(1-s)s^{2\alpha} \nonumber \\
	&\vphantom{\int_{S'_{1}}} \approx s^{p\alpha-p+2}+s^2+s^{2+(1-\alpha) q}+s^{2\alpha+1}. \nonumber
\end{align}
Setting $\gamma:=\min\{p\alpha - p +2, 2+(1-\alpha) q, 2\alpha+1\}$, since $\frac{p-1}{p}<\alpha<\beta \leq 1$ we have that $\gamma>1$. By~\eqref{e:estimate-S1} we conclude that
\begin{equation}
\label{e:energy-S1}
	\int_{S_{1}} W(\nabla y_{\alpha,\beta}) - W(\idmatrix) \, dx \lesssim s^\gamma.
\end{equation}

For $x \in S_{2}$, we extend~$y_{\alpha,\beta}$ with a suitable shifted copy. With a slight abuse of notation, we set
\begin{displaymath}
	y_{\alpha,\beta}(x):=Q y_{\alpha,\beta} (Q^{T}(x - \xi)),
\end{displaymath}
where $Q$ and~$\xi$ are given by~\eqref{def:omega_s}. It is straightforward that~$y_{\alpha,\beta}$ is injective on 
$\Omega_s\setminus (\{0\}\times (-s,s)\cup (4-s,4+s)\times \{0\})$
while $y_{\alpha,\beta}(\{0\}\times (-s,s)) = y_{\alpha,\beta}((4-s,4+s)\times \{0\}) = \{0\}$.
By the change-of-variables formula for Sobolev mappings, $y_{\alpha,\beta}$ satisfies \eqref{CN}. Clearly, the estimate~\eqref{e:energy-S1} holds true also on~$S_{2}$. This concludes the proof of~\eqref{e:step1} for $q \in (1, \frac{p}{p-2})$.   

To cover $q=\frac{p}{p-2}$ we need to consider a slightly different example. With the same notation introduced above for $S_{1}'$ and $S_{1}''$, we set   
\begin{displaymath}
	\hat{y}_{\alpha, \beta}   (x):=\vect{\displaystyle\frac{x_1}{\abs{x_1}^{1-\alpha} \abs{\ln{ | x_1 | }}} \\[5mm] 
	\displaystyle \abs{\frac{\ln{s}}{\ln{| x_1 | }}}^2  \cdot \frac{\abs{x_1}^{\beta}}{s^{\beta}} x_2}
	\quad \text{for} \: x\in S_1',
\end{displaymath}
and for $x\in S_1''$
\begin{displaymath}
	\hat{y}_{\alpha, \beta} (x):=\vect{\displaystyle \frac{x_{1}}{|x_{1}|} \, \left(\frac{1-\frac{s^\alpha}{\abs{\ln s}}}{1-s} ( |x_1| - 1) + 1\right) \\[4mm] x_2}.
\end{displaymath}

It is straightforward to check that in $S'_1$
\begin{displaymath}
	|\nabla \hat{y}_{\alpha, \beta} |^p \lesssim
	\frac{\abs{x_1}^{p(\alpha - 1)}}{  \abs{\ln{  | x_1 |   }}^p }  
	+ \frac{ \abs{x_1}^{p(\alpha - 1)}}{ \abs{\ln{  | x_1 |   }}^{2p}}  + \frac{|x_{1}|^{p(\beta-1)}  | \ln s|^{2p} }{s^{(\beta-1) p} | \ln |x_{1}| |^{2p}} + \frac{|x_{1}|^{p(\beta-1)}  | \ln s|^{2p} }{s^{(\beta-1) p} | \ln |x_{1}| |^{3p}}
	+ \frac{|x_{1}|^{\beta p} | \ln s|^{2p}}{ s^{\beta p} | \ln | x_{1}| |^{2p}}.   
\end{displaymath}
and
\begin{displaymath}
	(\det\nabla \hat{y}_{\alpha, \beta}   )^{-q} \lesssim \frac{s^{q\beta}}{|\ln s|^{2q}}  \abs{x_1}^{q(1-\alpha-\beta)} | \ln |x_{1}| |^{3q}.
\end{displaymath}
We have  
$(\det\nabla \hat{y}_{\alpha, \beta}   )^{-q}+|\nabla \hat{y}_{\alpha, \beta} |^p\in L^1(S_1)$ 
if 
$\beta \geq \alpha$, 
$(1-\alpha -\beta)q\geq -1$ and $p(\alpha-1)\geq -1$.
Moreover, if $\beta = \alpha = \frac{p-1}{p}$ and $q=\frac{p}{p-2}$
\begin{align*}
	\int_{S_{1}} W(\nabla y_{\alpha,\beta}) - W(\idmatrix) \, dx & 
	\lesssim  
 o(s). 
\end{align*}
On $S_{1}''$ we can repeat the argument of~\eqref{W-upperboundrest}. This concludes the proof of the proposition.   
\end{proof}

\begin{proof}[Proof of Proposition~\ref{p:lav-2d-2}]
With an explicit construction of a competitor $y \in W^{1, \infty}(\Om_{s}; \RR^{2}) \cap \cY_{s}$ as
\begin{displaymath}
	y(x) := \left\{ 
	\begin{array}{ll}
		x & \text{for $x \in S_{1}$,}\\[2mm]
		x - \xi + 
		\left(
		\begin{matrix}
			\frac{1+s}{1-2s}(1-|x_2|)\\
			0
		\end{matrix}
		\right) 
		& \text{for $x \in S_{2}$},
	\end{array}\right.
\end{displaymath}
one can show that there exists $M>0$ such that
\begin{displaymath}
	\min_{y\in W^{1,\infty} (\Om_{s}; \RR^{d}) \cap \cY_s} E_s(y)\leq Ms .
\end{displaymath}

Let us now fix $y \in W^{1, \infty}(\Om_{s}; \RR^{2}) \cap \cY_{s}$ with finite energy~$E_{s} (y)$. We claim that for a.e.~$\sigma \in (-s, s)$ one of the following \RRR inequalities are \EEE satisfied:
\begin{subequations}
\label{neq:stretching}
	\begin{align}
		& \int_{-1}^{1}  | \partial_{x_{1}} y(x_{1}, \sigma) | \, dx_{1}  \geq  2\sqrt{2} ( 1 - s), \label{e:either-1}\\
		& \int_{-1}^{1}  | \partial_{x_{2}} y(4 + \sigma, x_{2}) | \, dx_{2}  \geq  2\sqrt{2} ( 1 - s).  \label{e:either-2}
	\end{align}
\end{subequations}

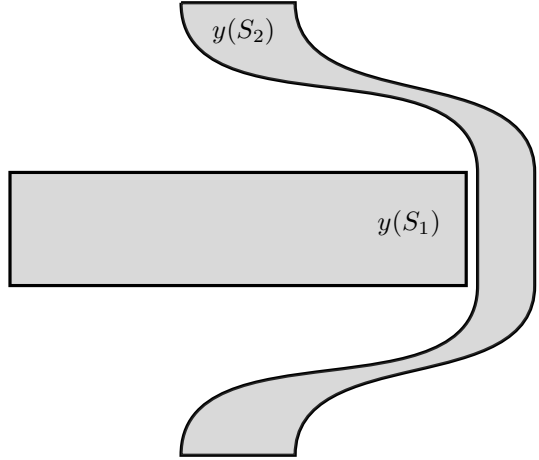
\begin{figure}[h!]
	\begin{tikzpicture}[scale = 1.5, every node/.style={font=\footnotesize}]
		\filldraw[very thick, fill= gray, fill opacity=0.3] (-2,-0.5) rectangle (2,0.5);
		\draw[very thick, smooth] (-0.5,2) to [in=90,out=-90] (2.1,0.5) -- (2.1,-0.5) to [in=90,out=-90] (-0.5,-2) -- (0.5,-2) to [in=-90,out=90] (2.6,-0.5) -- (2.6,0.5) to [in=-90,out=90] (0.5,2) -- (-0.5,2);
		\fill[fill= gray, fill opacity=0.3] (-0.5,2) to [in=90,out=-90] (2.1,0.5) -- (2.1,-0.5) to [in=90,out=-90] (-0.5,-2) -- (0.5,-2) to [in=-90,out=90] (2.6,-0.5) -- (2.6,0.5) to [in=-90,out=90] (0.5,2);

		\node at (1.5, 0.05) {$y(S_{1})$};
		\node at (0.05, 1.75) {$y(S_{2})$};
	\end{tikzpicture}
	\caption{Injective everywhere deformation.}
	\label{f:around}
\end{figure}

For $\sigma \in (-s, s)$, let us denote by $T_1^{\sigma} := (-1, 1) \times \{\sigma\}$ and $T^{\sigma}_2:= \{4 + \sigma\} \times (-1, 1)$ the sections of each stripe. By the boundary conditions and continuity of $y$, for every $\sigma, \zeta \in (-s, s)$ the curve $y(T^{\sigma}_1)$ has to intersect the line $\{z\in\RR^2\mid z_1=\zeta\}$. Similarly, $y(T^{\sigma}_{2})$ has to intersect  $\{z\in\RR^2\mid z_2 = \zeta\}$ (see also Fig.~\ref{f:1}). For $\sigma \in (-s, s)$, we distinguish two cases: 
\begin{enumerate}[label=(\roman*),ref=\roman*]
\item \label{i}$y(T^{\sigma}_1)$ intersects $\{\sigma\} \times ( (-\infty, -1] \cup [1, +\infty))$ or $y(T^{\sigma}_2)$ intersects  $( (-\infty, -1] \cup [1, +\infty)) \times \{\sigma\}$;
\item \label{ii} $y(T^{\sigma}_1)$ and $y(T^{\sigma}_2)$ only intersect $\{\sigma\} \times (-1, 1)$ and $(-1, 1) \times \{\sigma\}$, respectively.
\end{enumerate}

Denoting by~$K (x,y(x) )$ the distorsion of~$y\in W^{1, \infty}(\Om_{s}; \RR^{2}) \cap \cY_{s}$ in $x \in \Om_{s}$ 
\begin{equation}\label{distortion2d}
K (x,y(x) ) := \frac{| \nabla y (x) |^{2}}{(\det \nabla y(x))}, 
\end{equation}
we notice that~$y$ satisfies 
\begin{displaymath}
	\int_{\Omega_s} K^q (x,y) \, dx  \lesssim \|\nabla y\|_{L^\infty}^{2q} \, E_s(y) <\infty,
\end{displaymath}
with $q>d-1=1$. Since~$y$ is nonconstant, due to the boundary data, by the Reshetnyak theorem~\cite{ManVill1998} for mappings of finite distortion, $y$ is open and discrete. Moreover, any open map that is injective almost everywhere is indeed injective everywhere
(as pointed out in \cite[Lemma 3.3]{GraKruMaiSte19a}). Hence,  the case~\eqref{ii} is impossible, 
and the general deofrmation is pictured in Fig.~\ref{f:around}.

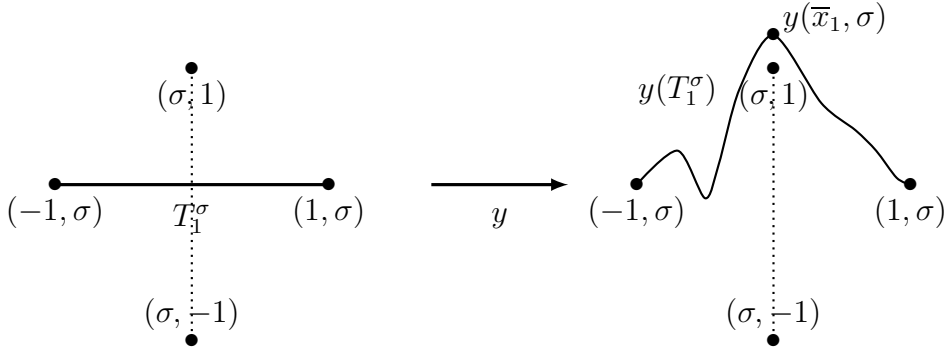
\begin{figure}[h!]
	\begin{tikzpicture}[xscale=0.9,yscale=0.9, >=latex]
		\draw[black, very thick] (0,0.3) -- (4,0.3);
		\draw[black, thick, dotted] (2,2) -- (2,-2);
		\node at (0, -0.1) {$(-1, \sigma)$};
		\node at (4, -0.1) {$(1, \sigma)$};
		\node at (2, -0.2) {$T^{\sigma}_{1}$};
		\node at (2,2) {\textbullet};
		\node at (2,-2) {\textbullet};
		\node at (0,0.3) {\textbullet};
		\node at (4,0.3) {\textbullet};
		\node at (2, 1.6) {$(\sigma, 1)$};
		\node at (2, -1.6) {$(\sigma, -1)$};
		\draw[black, ->, very thick] (5.5,0.3) -- (7.5, 0.3);
		\node at (6.5, -0.2) {{$y$}};
		\node at (9.1, 1.7) {$y(T^{\sigma}_{1})$};
		\node at (8.5,0.3) {\textbullet};
		\node at (12.5,0.3) {\textbullet};
		\node at (10.5,2) {\textbullet};
		\node at (10.5,-2) {\textbullet};
		\draw[thick, smooth] plot coordinates
		            {
		                (8.5, 0.3)
		                (9.1, 0.8)
		                (9.5, 0.1)
		                (9.7, 0.6)
		                (10, 1.7)
		                (10.5, 2.5)
		                (11.2, 1.5) 
		                (11.7, 1.1)
		                (12, 0.8)
		                (12.3, 0.4)
		                (12.5, 0.3) 
		            };
		\draw[black, thick, dotted] (10.5,2) -- (10.5,-2);
		\node at (10.5, 1.6) {$(\sigma, 1)$};
		\node at (10.5, -1.6) {$(\sigma, -1)$};
		\node  at (8.5, -0.1) {$(-1, \sigma)$};
		\node  at (12.5, -0.1) {$(1, \sigma)$};
		\node at (10.5, 2.5) {\textbullet};
		\node at (11.4, 2.75) {$y (\overline{x}_{1}, \sigma)$};
	\end{tikzpicture}
	\caption{Graphic representation of $y(T^{\sigma}_{1})$ satisfying \eqref{i}.} \label{f:1}
\end{figure}

Therefore, for every $\sigma \in (-s, s)$ we are in the case~\eqref{i}.
For every $\sigma \in (-s, s)$ such that the integrals in~\eqref{neq:stretching} are well defined, we may assume without loss of generality that~$y(T^{\sigma}_{1}) \cap [ \{\sigma\} \times [1, +\infty) ] \neq \varnothing$ (the other cases can be treated similarly), and let 
$\overline{x}_{1} \in (-1, 1)$ be such that $y( \overline{x}_{1}, \sigma)  \in y(T^{\sigma}_{1}) \cap [ \{\sigma\} \times [1, +\infty)]$. Since the shortest path connecting $(-1, \sigma)$ to the point~$y( \overline{x}_{1}, \sigma)$ is the segment, by the boundary conditions of~$y$ we have that 
\begin{align}
	\label{e:length}
	\int_{-1}^{\overline{x}_{1}} | \partial_{x_{1}} y(x_{1}, \sigma) | \, dx_{1} & \geq  \sqrt{ 2} (1 - \sigma) \geq \sqrt{2} ( 1 - s).  
\end{align}
With the same argument, we deduce that
\begin{align}
	\label{e:length-2}
	\int_{\overline{x}_{1}}^{1} | \partial_{x_{1}} y(x_{1}, \sigma) | \, dx_{1} & \geq  \sqrt{2} ( 1 - s). 
\end{align} 
Combining~\eqref{e:length}--\eqref{e:length-2} we obtain~\eqref{e:either-1}. If $y(T^{\sigma}_{2})$ intersects $((-\infty, -1] \cup [1, +\infty)) \times \{\sigma\}$, the same argument leads to~\eqref{e:either-2}.

We are now in a position to conclude for~\eqref{e:step2}. We define the sets
\begin{align*}
	A &:= \mysetr{ x_2 \in (-s,s) }{ \text{\eqref{e:either-1} is satisfied}},\\
	B & := \mysetr{ x_{1} \in (-s, s) \setminus A}{ \text{\eqref{e:either-2} is satisfied}}.
\end{align*}
In view of~\eqref{neq:stretching}, we have that $A \cup B = (-s, s)$, up to a set of $\mathcal{L}^{1}$-measure zero. Moreover, $A \cap B = \varnothing$. By~\eqref{Wlowerbound} we estimate (recall that $\| \cdot\|_{2}$ denotes the operator norm)
\begin{align}
\label{e:en-estimate}
	\int_{\Omega_{s}} & (W(\nabla y) - W( \idmatrix) )\,  d x 
	\\
	&
	= \int_{-s}^{s} \int_{-1}^{1} (W(\nabla y) - W(\idmatrix) )\, d x_{1} \, dx_{2} +  \int_{4-s}^{4+s} \int_{-1}^{1} (W(\nabla y)  - W(\idmatrix) )\, d x_{2} \, dx_{1} \nonumber
	\\
	&
	\geq c \int_{A} \int_{-1}^{1}  \abs{\norm{\nabla y}_2-1}^p +  (\norm{\nabla y}_2-1)^2\, d x_{1} \, dx_{2} \nonumber
	\\
	&
	\qquad + c \int_{4 + B} \int_{-1}^{1}  \abs{\norm{\nabla y}_2-1}^p + (\norm{\nabla y}_2-1)^2\,d x_{2} \, dx_{1}. \nonumber
	\end{align}
	Thanks to the Jensen inequality, to~\eqref{neq:stretching}, and to the definition of~$A$ and~$B$, we continue in~\eqref{e:en-estimate} with
	\begin{align*}
	\int_{\Omega_{s}} & (W(\nabla y) - W( \idmatrix) ) d x 
	\\
	&
	\geq c \int_{A} \int_{-1}^{1}  \abs{| \partial_{x_{1}} y| - 1}^p + (| \partial_{x_{1}} y| - 1)^2\, d x_{1} \, dx_{2} \nonumber
	\\
	&
	\qquad + c \int_{4 + B} \int_{-1}^{1}  \abs{| \partial_{x_{2}} y| - 1}^p + (| \partial_{x_{2}} y| - 1)^2\,d x_{2} \, dx_{1} \nonumber
	\\
	&
	\geq c \int_{A} \abs{\int_{-1}^{1} | \partial_{x_{1}} y|  \, dx_{1} - 1}^{p} + \Big( \int_{-1}^{1} | \partial_{x_{1}} y| \, dx_{1} - 1 \Big)^{2} \, dx_{2} \nonumber
	\\
	&
	\qquad + c \int_{4 + B}  \abs{ \int_{-1}^{1}  | \partial_{x_{2}} y|\, dx_{2}  - 1}^p + \Big(\int_{-1}^{1} | \partial_{x_{2}} y| \, dx_{2}  - 1\Big)^2  \, dx_{1} \nonumber
	\\
	&\geq \,c\,  (|A| + |B|)  \big[  \big( 2\sqrt{2} ( 1 - s)  - 1\big)^{p} +\big( 2\sqrt{2} ( 1 - s ) - 1 \big)^{2} \big] \geq  m s + o(s)
\end{align*}
for some positive constant~$m$ independent of~$y$ and of~$s$. This concludes the proof of~\eqref{e:step2}.
\end{proof}

\begin{rem}\label{rem:1}
	The Lavrentiev phenomenon is valid even if we replace $W^{1, \infty} (\Om_{s}; \RR^{2})$ with $W^{1,r} (\Om_{s}; \RR^{2})$ for $r>\frac{2q}{q-1}$. 
	In this case, we have that for $y \in W^{1,r} (\Om_{s}; \RR^{2})\cap \cY_{s}$ with $E_{s} (y) <+\infty$,  the distortion coefficient $K(x,y)$ defined in \eqref{distortion2d}, belongs to~$L^{\eta} (\Om_{s})$ for~$\eta:=\frac{rq}{2q+r}$, $\eta \in (1, q)$. Indeed, by H\"older inequality it holds
	\begin{equation}\label{neq:distprtion-int}
		\int_{\Omega_{s}} \left(\frac{|\nabla y  |^2}{\det \nabla y  }\right)^\eta \, dx \leq 
		\left(\int_{\Omega_{s}} |\nabla y  |^{2\eta\frac{q}{q-\eta}} \, dx\right)^{\frac{q-\eta}{q}} 
		\left(\int_{\Omega_{s}} \frac{dx}{\det \nabla y ^{\eta \frac{q}{\eta}}}\right)^{\frac{\eta}{q}} 
		<\infty,
	\end{equation}
	since $r=\frac{2q\eta}{q-\eta}$ and $E_{s} (y) <+\infty$. This implies that any competitor $y \in W^{1,r} (\Om_{s}; \RR^{2})\cap \cY_{s}$ with finite energy must satisfy~\eqref{i} for every $\sigma \in (-s, s)$. Then, the proof of the lower bound of~$E_{s}(y)$ proceeds as in the $W^{1,\infty}$-case.
\end{rem}

\begin{rem}\label{rem:2}
	The argument in Remark~\ref{rem:1} also shows that the two-dimensional example in Proposition~\ref{p:lav-2d-1} is optimal in the following sense: if $p>2$ and $q > \frac{p}{p-2}$, then every $y \in W^{1, p} (\Om_{s}; \RR^{2}) \cap \cY_{s}$ with finite energy satisfies~$K_{y} \in L^{\eta} (\Om_{s})$ for $\eta =   \frac{pq}{2q+p} > 1$ (see  \eqref{neq:distprtion-int}). Hence, $y$ has to be injective. This would rule out the example constructed in the proof of Proposition~\ref{p:lav-2d-1}.
\end{rem}

\section{The Lavrentiev phenomenon in dimension three}
\label{sec:Lav3d}

In this section, we show a three-dimensional generalization of the Lavrentiev phenomenon proven in Theorem~\ref{thm:lav2d}. 
  The example is created by simply thickening the two-dimensional version in another direction, corresponding to the variable $x_1$ below, 
while $(x_2,x_3)$ correspond to the two variables of the 2D example.   

For $s \in (0, 1)$, the reference configuration~$\Omega_{s}$ consists now of the union of two thin cuboids of width~$s$. Namely, we write
\begin{equation}\label{def:omega_s-3D}
\begin{split}
	\Omega_s & :=S_1 \cup S_2,\\
	S_1 & :=(-1,1)\times (-1, 1) \times (-s,s) , \qquad S_2:=\xi + Q S_1\, , \\
	Q & :=\left(\begin{array}{rrr}1 & 0 & 0\\ 0 & 0 & -1 \\ 0 & 1 & 0 \end{array}\right), \qquad  \xi:=\vect{0 \\ 4 \\ 0}.
\end{split}
\end{equation}
We consider the Dirichlet datum
\begin{displaymath}
	y_0(x):=\begin{cases}
	x & \text{for $x\in \overline{S_1}$,}\\
	x-\xi &\text{for $x\in \overline{S_2}$,}
	\end{cases}
\end{displaymath}
and the set of admissible deformations
\begin{displaymath}
	\cY_s:=\mysetr{y\in W_+^{1,p}(\Omega_s;\RR^3)}{\text{\eqref{CN} holds and $y=y_0$ on $\Gamma_s$}},
\end{displaymath}
where
\begin{displaymath}
	\Gamma_s:= \big( [-1, 1] \times \{-1,1\}\times [-s,s] \big) \cup \big( [-1, 1] \times [4-s,4+s] \times \{-1,1\} \big) \subset \partial\Omega_s.
\end{displaymath}

Similar to Theorem~\ref{thm:lav2d}, we have the Lavrentiev phenomenon in the following form.

\begin{theorem}
	\label{thm:lav3d}
	For every $p \in (3, 4)$ and every $q \in (2, \frac{p}{p-2})$ there exists $\overline{s} \in (0, 1]$ such that for every $s \in (0, \overline{s}]$ the following holds:
	\begin{align}
	\label{e:l2-3d}
		\inf_{y \in W^{1, \infty}(\Omega_{s}; \RR^{3}) \cap \cY_{s}} \, E_{s}(y) > \inf_{y \in \cY_{s}} \,  E_{s}(y) . 
	\end{align}
\end{theorem}

The proof of Theorem~\ref{thm:lav3d} is subdivided into two propositions given below. Compared to the two-dimension case, 
we now have to face an additional difficulty, because ``fully going around'' (case \eqref{i} in the proof of  Proposition~\ref{p:lav-2d-2}) is no longer the only way the two pieces can avoid each other after deformation. 
In principle, it should be possible to generalize our three-dimensional example to any dimension $d\geq 3$, but for the sake of simplicity, we 
will stick to $d=3$, the practically most relevant case.

\begin{prop}
\label{p:lav3d-1}
	For every $p \in (3, 4)$ and every $q \in (2, \frac{p}{p-2})$, 
	$\inf_{y \in \cY_{s}} \,  E_{s}(y)=o(s)$, i.e,
	\begin{equation}
	\label{e:3d-step1} 
		\lim_{s\searrow 0} \frac{1}{s}\inf_{y \in \cY_{s}} \,  E_{s}(y) =0.   
	\end{equation}
\end{prop}

\begin{prop}
\label{p:lav3d-2}
	For every $p \in (3, 4)$ and every $q \in (2, \frac{p}{p-2})$, 
	there exists $\overline{s} \in (0, 1)$ such that for every $s \in (0, \overline{s}]$
	\begin{equation}
	\label{e:3d-step2}  
		ms\leq \inf_{y \in W^{1, \infty}(\Omega_{s}; \RR^{3}) \cap \cY_{s}} \, E_{s}(y) \leq Ms
	\end{equation}
	with constants $0<m<M<+\infty$ independent of $s$.
\end{prop}

We start with the proof of Proposition~\ref{p:lav3d-1}.

\begin{proof}[Proof of Proposition~\ref{p:lav3d-1}]
As in the proof of Proposition~\ref{p:lav-2d-1}, it is enough to construct a sequence of competitors~$y^{s} \in \cY_{s}$ satisfying $E_{s} (y^{s}) =o(s)$ as $s \searrow 0$. To this purpose, let us fix $\alpha, \beta \in (0, 1)$ (to be determined later on) and let us define $S'_{1}:= \{ x \in S_{1}: | x_{2} | \leq s\}$, $S'_{2}:= \{x \in S_{2}: | x_{3}| \leq s\}$, and $S''_{i} := S_{i} \setminus S'_{i}$ for $i = 1, 2$. 

In order to prove the asymptotic~\eqref{e:3d-step1} we define the map $y_{\alpha, \beta}\colon \Omega_{s} \to \RR^{3}$ as
\begin{align*}
	y_{\alpha, \beta} (x) & := \left( \begin{array}{ccc}
	x_{1}
	\\[1mm]
	\displaystyle \frac{x_{2}}{| x_{2} |^{1 - \alpha}}
	\\[5mm]
	\displaystyle\frac{|x_{2}|^{\beta}}{s^{\beta}} \, x_{3}
	\end{array}\right) \qquad \text{for $x \in S'_{1}$},\\[2mm]
	y_{\alpha, \beta} (x) & :=   \left( \begin{array}{ccc}
	x_{1}
	\\[2mm]
	\displaystyle \frac{x_{2}}{|x_{2}|} \, \left( \frac{1-s^\alpha}{1-s} ( | x_{2} | - 1 ) + 1\right) 
	\\[2mm]
	 x_{3}
	\end{array}\right) \qquad \text{for $x \in S''_{1}$},\\[2mm]
	y_{\alpha, \beta} (x) & :=Q y_{\alpha, \beta} (Q^{T}( x - \xi ) ) \qquad \text{for $x \in S_{2}$}.
\end{align*}
To show that $y_{\alpha, \beta} \in \cY_{s}$ for $s$ small, we have to show that $\nabla y_{\alpha, \beta} \in L^{p}(\Omega_{s}; \RR^{3\times 3})$. We  focus on~$S_{1}$, as the definition of~$y_{\alpha, \beta}$ leads to the same computations on~$S_{2}$. By construction of $y_{\alpha, \beta}$, on~$S_{1}$ we have that
\begin{align*}
	\nabla y_{\alpha, \beta} (x) & = \left( 
	\begin{array}{ccc}
		1 & 0 & 0\\
		0 & \alpha | x_{2}|^{\alpha - 1} & 0\\
		0 & \frac{\beta}{s^{\beta}} |x_{2}|^{\beta - 2} x_{2} x_{3} & \frac{|x_{2}|^{\beta}}{s^{\beta}} 
	\end{array}\right) \qquad \text{for $x \in S'_{1}$},\\[2mm]
	\nabla y_{\alpha, \beta} (x) & = \left( 
	\begin{array}{ccc}
		1 & 0 & 0\\
		0 & \frac{1 - s^{\alpha}}{1 - s} & 0\\
		0 & 0 & 1 
	\end{array}\right) \qquad \text{for $x \in S''_{1}$}.\\
\end{align*}
Imposing $\nabla y_{\alpha, \beta} \in L^{p}(S_{1};\RR^{3\times 3})$ implies that 
\begin{align}
\label{e:alphabeta}
	\alpha, \beta > 1 - \frac{1}{p}.
\end{align}

We notice that 
\begin{align*}
	\det \nabla y_{\alpha, \beta} (x)  & = \frac{\alpha  | x_{2}|^{\alpha + \beta - 1} }{s^{\beta} }\qquad \text{in $S'_{1}$} ,\\
	\det \nabla y_{\alpha, \beta}(x)  & =  \frac{ 1- s^{\alpha} }{1 - s}\qquad \text{in $S''_{1}$},
\end{align*}
so that $\det \nabla y_{\alpha, \beta} >0$ on $\Omega_{s}$. As in the proof of Theorem~\ref{thm:lav2d}, $y_{\alpha, \beta}$ is injective on $\Omega_{s} \setminus \big( (-1, 1) \times \{0 \} \times (-s,s) \cup (-1, 1) \times (4-s,4+s) \times \{0 \} \big)$ and
while $y_{\alpha, \beta} ((-1, 1) \times \{0 \}\times (-s,s)) = y_{\alpha, \beta} ((-1, 1) \times (4-s,4+s)\times \{0\}) = (-1, 1) \times \{0\} \times \{0\}$.
Thus, $y_{\alpha, \beta}$ satisfies \eqref{CN} and $y_{\alpha, \beta} \in \cY_{s}$ for $\alpha, \beta \in (0, 1)$ such that~\eqref{e:alphabeta} holds.

 Imposing the integrability of $ (\det \nabla y_{\alpha, \beta})^{-q}$ on~$S_{1}$ we deduce that it must be
\begin{align}
\label{e:alphabeta-2}
	( 1- \alpha - \beta ) q >-1.
\end{align}
Combining ~\eqref{e:alphabeta} and~\eqref{e:alphabeta-2}, we infer that for any choice of~$p \in (3, 4)$ and of~$q \in (2, \frac{p}{p-2})$, we can find $\alpha, \beta \in (0, 1)$ such that $y_{\alpha, \beta} \in \cY_{s}$ with $(\det \nabla y_{\alpha, \beta} )^{-q} \in L^{1}(\Omega_{s})$. 
A direct estimate of $W(\nabla y_{\alpha, \beta})$ on~$S'_{1}$ yields that
\begin{align}
\label{e:W3d}
	0 &\leq \int_{S'_{1}} W(\nabla y_{\alpha, \beta}) - W(\idmatrix) \, dx 
	\lesssim s^{2} + s^{(\alpha -1)p + 2}  + s^{\beta q + 1} s^{(1 - \alpha - \beta) q + 1} . 
\end{align}  
From~\eqref{e:alphabeta}--\eqref{e:W3d} we deduce that there exists $\rho \in (0, 1)$ (depending on~$\alpha, \beta$ but not on~$s$) such that 
\begin{align}
\label{e:W3d-2}
	0 \leq \int_{S'_{1}} W(\nabla y_{\alpha, \beta}) - W(\idmatrix) \, dx & \lesssim s^{1 + \rho}.
\end{align} 
As for $S''_{1}$, we may use the estimate of~\eqref{W-upperboundrest} and obtain that
\begin{equation}
\label{e:W3d-3}
	0 \leq \int_{S''_{1}} W(\nabla y_{\alpha, \beta}) - W(\idmatrix) \, dx  \lesssim s^{1 + 2\alpha}.
\end{equation}
Defining~$\delta := \min\{ \rho, 2\alpha\}$ we infer that
\begin{equation}
\label{e:W3d-4}
	0 \leq \int_{S_{1}} W(\nabla y_{\alpha, \beta}) - W(\idmatrix) \, dx  \lesssim s^{1 + \delta}.
\end{equation}
Arguing in the same way, estimate~\eqref{e:W3d-4} can be obtained on~$S_{2}$, leading to~\eqref{e:3d-step1}. This concludes the proof of the proposition.
\end{proof}

The following two lemmas show some useful properties of deformations~$y \in W^{1, \infty} (\Omega_{s};\RR^{3}) \cap \cY_{s}$ with low energy, which will be useful to conclude for~\eqref{e:3d-step2}. In the sequel, we denote by
$\pi\colon \RR \to [-1, 1]$ the projection of~$\RR$ to the interval $[-1, 1]$, defined as 
\begin{equation*}
	\pi(t):= 
	\begin{cases}
		t, & \text{if } -1\leq t\leq 1,\\
		-1, & \text{if } t< -1,\\
		1, & \text{if } t> 1.\\
	\end{cases}
\end{equation*}

\begin{lem}
\label{p:proof2}
	There exists $M>0$ such that for every $s \in (0, 1)$
	\begin{equation}
	\label{e:rough-estimate}
		\min_{y \in W^{1, \infty}(\Omega_{s};\RR^{3}) \cap \cY_{s}} \, E_{s} (y_{s}) \leq Ms.
	\end{equation}
\end{lem}

\begin{proof}
The thesis follows easily by a direct construction of a competitor $y \in W^{1, \infty}(\Omega_{s} ; \RR^{3}) \cap \cY_{s}$. For instance, we define
\begin{displaymath}
	y(x) := \left\{ \begin{array}{ll}
	x & \text{for $x \in S_{1}$,}\\[2mm]
	x - \xi + 
	\left(
	\begin{matrix}
		0\\
		\frac{1+s}{1-2s}(1-|x_3|)\\
		0
	\end{matrix}
	\right) 
	& \text{for $x \in S_{2}$}.
	\end{array}\right.
\end{displaymath}
Then, it is clear that $E_{s} (y) \leq Ms$ for some $M>0$ independent of~$s$.  
\end{proof}

\begin{lem}
\label{p:proof3}
	Let $s \in (0, 1)$, $N>0$, $\gamma := 1 - \frac{2}{p}$, $\sigma \in (-s, s)$, and $y \in W^{1, \infty} (\Omega_{s};\RR^{3}) \cap \cY_{s}$ be such that 
	\begin{align}
		& \int_{-1}^{1} \int_{-1}^{1} W(\nabla y( x_{1}, x_{2}, \sigma)) - W(\idmatrix) \, dx_{1} \, dx_{2} \leq N \label{e:EN}.
	\end{align}
	Then, there exists $\overline{c}_{N, p}>0$ depending only on~$p$ and~$N$ (but not on~$s$) such that for every $\varepsilon>0$ the following holds: if
	\begin{align}
		& \int_{-1}^{1} \Bigg| \int_{-1}^{1} \big( | \partial_{2} y | (x_{1}, x_{2}, \sigma) - 1\big) dx_{2}  \Bigg|^{p} dx_{1} \leq \varepsilon , \label{e:eps}
	\end{align}
	then for every $x_{1} , x_{2} \in [-1, 1]$ 
	\begin{equation}
	\label{e:distance-H}
		\big| y (x_{1}, x_{2} , \sigma) - (x_{1}, \pi (y_{2} (x_{1}, x_{2} , \sigma) ) , \sigma) \big|  \leq  \overline{c}_{N, p} \, \varepsilon^{\frac{\gamma}{p+1}}.
	\end{equation}

	Similarly, if 
	\begin{align}
		& \int_{-1}^{1} \int_{-1}^{1} W(\nabla y( x_{1}, \sigma + 4, x_{3})) - W(\idmatrix) \, dx_{1} \, dx_{3} \leq N, \label{e:EN-2} \\
		& \int_{-1}^{1} \Bigg| \int_{-1}^{1} \big( | \partial_{3} y | (x_{1}, \sigma + 4, x_{3} ) - 1\big) dx_{3}  \Bigg|^{p} dx_{1} \leq \varepsilon , \label{e:eps-2}
	\end{align}
	then, for every $x_{1}$, $x_{3} \in [-1, 1]$
	\begin{equation}
	\label{e:distance-H-2}
		\big| y (x_{1},  \sigma + 4, x_{3} ) - (x_{1}, \sigma , \pi (y_{3} (x_{1}, \sigma + 4, x_{3}) ) ) \big|  \leq  \overline{c}_{N, p} \, \varepsilon^{\frac{\gamma}{p+1}}.
	\end{equation}
\end{lem}

\begin{proof}
As $p>3$, Morrey's embedding and~\eqref{e:EN} imply that the map $(x_{1}, x_{2}) \mapsto y(x_{1}, x_{2}, \sigma)$ is H\"older-continuous. Precisely, there exists $\widetilde{c}_{N, p}>0$ depending only on~$p$ and $N$ such that for every $x_{1}, x_{2}, \overline{x}_{1}, \overline{x}_{2} \in [-1, 1]$
\begin{equation}
\label{e:Holder}
	| y(x_{1}, x_{2}, \sigma) - y(\overline{x}_{1}, \overline{x}_{2}, \sigma)| \leq \widetilde{c}_{N, p}  \, | (x_{1}, x_{2}) - (\overline{x}_{1}, \overline{x}_{2})|^{\gamma}.
\end{equation}

Let us define the set $D_{\varepsilon} \subseteq (-1, 1)$ as
\begin{displaymath}
	D_{\varepsilon}:= \Bigg\{ x_{1} \in (-1, 1) : \ \int_{-1}^{1} \big( | \partial_{2} y| (x_{1}, x_{2}, \sigma) - 1\big) \, dx_{2} \leq \varepsilon^{\frac{1}{p+1}}\Bigg\}.
\end{displaymath}
In particular, we notice that, due to the boundary condition on~$\Gamma_{s}$, we have that
\begin{displaymath}
	\int_{-1}^{1} \big( | \partial_{2} y| (x_{1}, x_{2}, \sigma) - 1\big) \, dx_{2} \geq 0 \qquad \text{for a.e.~$x_{1} \in (-1, 1)$}.
\end{displaymath}
Then, by~\eqref{e:eps} and by the Chebyshev inequality, 
\begin{align}
\label{e:Deps}
	\mathcal{H}^{1} & ((-1, 1) \setminus D_{\varepsilon})  
	\leq \frac{1}{\varepsilon^{\frac{p}{p+1}}} \int_{-1}^{1} \Bigg| \int_{-1}^{1} \big( | \partial_{2} y | (x_{1}, x_{2}, \sigma) - 1 \big) \, dx_{2} \Bigg|^{p} dx_{1} 
	\leq \varepsilon^{\frac{1}{p+1}}. 
\end{align}

Let us now fix $x_{1} \in D_{\varepsilon}$ and $x_{2} \in [-1, 1]$, let us denote by $\theta_{x_{1}} \colon [-1, 1] \to \RR^{3}$ the curve $\theta_{x_{1}} (t) := y(x_{1}, t, \sigma)$, and let us write
\begin{displaymath}
	y(x_{1}, x_{2}, \sigma) = (x_{1}, \pi ( y_{2} (x_{1}, x_{2}, \sigma)) , \sigma) + v.
\end{displaymath}
for suitable $v = (v_{1}, v_{2}, v_{3}) \in \RR^{3}$ depending on~$(x_{1}, x_{2})$. In particular, we notice one of the two cases must hold:
\begin{enumerate}[label=(\roman*), ref=\roman*]
	\item \label{3di} $\pi ( y_{2} (x_{1}, x_{2}, \sigma)) =   y_{2} (x_{1}, x_{2}, \sigma)$ and $v_{2} = 0$,
	\smallskip
	\item \label{3dii} $\pi ( y_{2} (x_{1}, x_{2}, \sigma)) \in \{1, -1\}$ and $|v_{2}| = \min \{ | 1 -   y_{2} (x_{1}, x_{2}, \sigma)| ; | 1+ y_{2} (x_{1}, x_{2}, \sigma)|\} $.
\end{enumerate}
In the case~\eqref{3di}, by definition of~$D_{\varepsilon}$ and by the boundary conditions on~$y$ we have that
\begin{align}
\label{e:something}
	2 + \varepsilon^{\frac{1}{p+1}} & \geq \int_{-1}^{1} | \dot{\theta}_{x_{1}} (t)| \, dt \geq |(x_{1}, -1, \sigma) - y(x_{1}, x_{2}, \sigma)| + |  y(x_{1}, x_{2}, \sigma) - (x_{1}, 1, \sigma)| 
	 \\
	 &
	 = | (v_{1}, y_{2} (x_{1}, x_{2}, \sigma) + 1, v_{3})| + | (v_{1}, y_{2} (x_{1}, x_{2}, \sigma) - 1, v_{3})| \nonumber
	 \\
	 &
	 = \sqrt{(v_{1}^{2} + v_{3}^{2}) + ( y_{2} (x_{1}, x_{2}, \sigma) + 1)^{2}}  + \sqrt{(v_{1}^{2} + v_{3}^{2}) + ( y_{2} (x_{1}, x_{2}, \sigma) - 1)^{2}}\nonumber
	 \\
	 &
	 \geq \min_{z \in [-1, 1]}  \sqrt{(v_{1}^{2} + v_{3}^{2}) + ( z + 1)^{2}}  + \sqrt{(v_{1}^{2} + v_{3}^{2}) + ( z- 1)^{2}} \nonumber
	 \\
	 &
	 = 2 \sqrt{ (v_{1}^{2} + v_{3}^{2}) + 1}, \nonumber
\end{align}
which implies that $|v| \leq \varepsilon^{\frac{1}{p+1}} +  \varepsilon^{\frac{1}{2 (p+1)}}$.

If \eqref{3dii} holds and $y_{2}( x_{1}, x_{2}, \sigma) \notin [-1, 1]$, we may repeat the argument of the first two lines of~\eqref{e:something} and obtain that $|v| \leq \varepsilon^{\frac{1}{p+1}}$. All in all, we have shown that for every $x_{1} \in D_{\varepsilon}$ and every $x_{2} \in [-1, 1]$ it holds
\begin{equation}
\label{e:intermediate-proof2}
	| y(x_{1}, x_{2}, \sigma) - (x_{1}, \pi (y_{2} (x_{1}, x_{2}, \sigma), \sigma ) | \leq  \varepsilon^{\frac{1}{p+1}} + \varepsilon^{\frac{1}{2 (p+1)}}.
\end{equation}

To achieve~\eqref{e:distance-H} it remains to consider $x_{1} \notin D_{\varepsilon}$. In this case, by~\eqref{e:Deps} we may find $\overline{x}_{1} \in D_{\varepsilon}$ such that $| x_{1} - \overline{x}_{1}| \leq 2 \varepsilon^{\frac{1}{p+1}}$. Then, by triangle inequality, by the H\"older continuity \eqref{e:Holder} of~$y$, and by the previous step we have that for every $x_{2} \in [-1, 1]$
\begin{align}
\label{e:something-2}
	|& y(x_{1}, x_{2}, \sigma) - (x_{1}, \pi (y_{2}(x_{1}, x_{2}, \sigma)), \sigma) |  
	\\
	&
	\leq | y(x_{1}, x_{2}, \sigma) - y( \overline{x}_{1}, x_{2}, \sigma)| + |y( \overline{x}_{1}, x_{2}, \sigma) -  (\overline{x}_{1}, \pi (y_{2}(\overline{x}_{1}, x_{2}, \sigma)), \sigma) | \nonumber
	\\
	&
	\quad + | (\overline{x}_{1}, \pi (y_{2}(\overline{x}_{1}, x_{2}, \sigma)), \sigma) -  (x_{1}, \pi (y_{2}(x_{1}, x_{2}, \sigma)), \sigma)| \nonumber
	\\
	&
	\leq 2 \, \widetilde{c}_{N, p} \,  | x_{1} - \overline{x}_{1}|^{\gamma} +  \varepsilon^{\frac{1}{p+1}} + \varepsilon^{\frac{1}{2 (p+1)}} + | x_{1} - \overline{x}_{1}|  \nonumber
	\\
	&
	\leq 2^{1+\gamma} \, \widetilde{c}_{N, p} \, \varepsilon^{\frac{\gamma}{p+1}} + 3 \varepsilon^{\frac{1}{p+1}} + \varepsilon^{\frac{1}{2 (p+1)}} \leq \overline{c}_{N, p} \varepsilon^{\frac{\gamma}{p+1}} \nonumber
\end{align}
for a suitable constant~$\overline{c}_{N, p}>0$ depending only on~$\gamma$ and~$\widetilde{c}_{N, p}$, and therefore only on~$p$ and~$N$. This concludes the proof of~\eqref{e:distance-H}.

The same argument can be used to infer~\eqref{e:distance-H-2} taking into account the boundary conditions on~$\partial S_{2}$.
\end{proof}

We are now in a position to prove Proposition~\ref{p:lav3d-2}.

\begin{proof}[Proof of Proposition~\ref{p:lav3d-2}]
Since Lemma~\ref{p:proof2} holds, we are left to provide a lower bound for the minimum problem~\eqref{e:3d-step2}. To this purpose, let $M >0$ be the constant determined in Lemma~\ref{p:proof2} and fix a deformation $y \in W^{1, \infty} (\Omega_{s};\RR^{3}) \cap \cY_{s}$ such that 
\begin{equation}
\label{e:en-bound-M}
	E_{s} (y) \leq (M+1)s.
\end{equation}

Let us fix $N >0$ such that $\frac{M+1}{N} <\frac{1}{10}$ and let us set
\begin{align*}
	A_{N} & := \Bigg\{ \sigma \in (-s, s) : \, \int_{-1}^{1} \int_{-1}^{1} W(\nabla y (x_{1}, x_{2}, \sigma) ) - W(\idmatrix) \, dx_{1} \, dx_{2} \leq N \Bigg\}, \\
	B_{N} & := \Bigg\{ \sigma \in (-s, s) : \, \int_{-1}^{1} \int_{-1}^{1} W(\nabla y (x_{1},  \sigma + 4, x_{3}) ) - W(\idmatrix) \, dx_{1} \, dx_{3} \leq N \Bigg\}.
\end{align*}
Then, by the Chebyshev inequlity and by~\eqref{e:en-bound-M} we have that
\begin{align}
	\mathcal{H}^{1} ((-s, s) \setminus A_{N}) & \leq \frac{1}{N} E_{s} (y) \leq \frac{M+1}{N} \, s < \frac{s}{10}, \label{e:AL}\\
	\mathcal{H}^{1} ((-s, s) \setminus B_{N}) & \leq \frac{1}{N} E_{s} (y) \leq \frac{M+1}{N} \, s < \frac{s}{10}. \label{e:BL}
\end{align}
Hence, we deduce from~\eqref{e:AL} and~\eqref{e:BL} that
\begin{equation}
\label{e:ALBL}
	\mathcal{H}^{1} (A_{N} \cap B_{N}) \geq \frac{18}{10}\, s.
\end{equation}
We further set~$\gamma:= 1 - \frac{2}{p}$ and fix $\varepsilon>0$ such that $\overline{c}_{N, p} \, \varepsilon^{\frac{\gamma}{p+1}} \leq \frac{1}{3}$, where~$\overline{c}_{N, p}>0$ is the constant defined in Proposition~\ref{p:proof3}. 
We claim that 
for every $\sigma \in A_{N} \cap B_{N}$, at least one of the following inequalities must  hold:
\begin{subequations}
\label{e:claim-3d}
	\begin{align}
		& \int_{-1}^{1} \Bigg| \int_{-1}^{1} \big( | \partial_{2} y| (x_{1}, x_{2}, \sigma) - 1\big) \, dx_{2} \Bigg|^{p} dx_{1} >\varepsilon, \label{e:claim-eps-1}\\
		& \int_{-1}^{1} \Bigg| \int_{-1}^{1} \big( | \partial_{3} y| (x_{1}, \sigma + 4, x_{3}) - 1\big) \, dx_{3} \Bigg|^{p} dx_{1} >\varepsilon. \label{e:claim-eps-2}
	\end{align}
	\end{subequations}
	By contradiction, let us assume that both inequalities~\eqref{e:claim-eps-1} and~\eqref{e:claim-eps-2} are not satisfied for some $\sigma$.
	By Lemma~\ref{p:proof3}, we deduce that for every $x_{1}, x_{2}, x_{3} \in [-1, 1]$,
	\begin{subequations}
	\label{e:y12}
	\begin{align}
		& | y(x_{1}, x_{2}, \sigma) - (x_{1}, \pi (y_{2}(x_{1}, x_{2}, \sigma)) , \sigma) | \leq \frac{1}{3},\label{e:y1}\\
		& | y(x_{1},  \sigma + 4, x_{3}) - (x_{1}, \sigma , \pi (y_{3}(x_{1}, \sigma + 4 , x_{3})) ) | \leq \frac{1}{3}.\label{e:y2}
	\end{align}
\end{subequations}
We now show that given \eqref{e:y1} and \eqref{e:y2}, $y\in \cY_{s}$ cannot be injective in $\Omega_s$. This immediately yields a contradiction, as we already know that any $y \in W^{1, \infty}(\Omega_{s}; \RR^{3}) \cap \cY_{s}$ with finite energy must be a homeomorphism (as a consequence of the theory mappings of finite distortion \cite{ManVill1998}, as already outlined in the introduction). 

To see that $y$ indeed cannot be injective, let us consider the function
\[
	g \colon [-1,1]^3\to \RR^3,\qquad g(x_1,\tau_1,\tau_2):=y(x_1,\tau_1,\sigma)-y(0,\sigma+4,-\tau_2). 
\]
Notice that if $g(\hat x_1,\hat\tau_1,\hat\tau_2)=0$ for some $(\hat x_1,\hat\tau_1,\hat\tau_2)\in(-1,1)^3$, then~$y$ is not injective, since $(\hat x_1,\hat\tau_1,\sigma)\in S_1$, $(0,\sigma+4,-\hat\tau_2)\in S_2$ and the values of~$y$ on these two points coincide. As a consequence of \eqref{e:y12} and of the boundary conditions of $y\in \cY_{s}$ on $\Gamma_s$, the vector field $g=(g_1,g_2,g_3)$ 
always points outwards on the boundary of the cube $[-1,1]^3$:   
\begin{displaymath}
\begin{alignedat}{2}
  & g_1(-1,\tau_1,\tau_2)\leq -1+\frac{2}{3}<0,\\
	& g_1(1,\tau_1,\tau_2)\geq 1-\frac{2}{3}>0,\\
	& g_2(x_1,-1,\tau_2)=-1-y_2(0,\sigma+4,-\tau_2)
	\leq -1-\sigma+\frac{1}{3}<0, \\
	& g_2(x_1,1,\tau_2)=1-y_2(0,\sigma+4,-\tau_2)
	\geq 1-\sigma-\frac{1}{3}>0,\\
	& g_3(x_1,\tau_1,-1)=y_3(x_1,\tau_1,\sigma)-1 \leq \sigma+\frac{1}{3}-1<0,\\
	& g_3(x_1,\tau_1,1)=y_3(x_1,\tau_1,\sigma)+1\geq 1-\sigma+\frac{1}{3}>0.\\
\end{alignedat}
\end{displaymath}
(Above, we also used that $s$ is small enough so that $\abs{\sigma}\leq s <\frac{1}{3}$.)
As $g$ is also continuous, it thus satisfies the prerequisites of the Poincar\'{e}--Miranda theorem (see \cite{Ma13a}, e.g.). The latter yields that $g$ attains the value $0\in\RR^3$ in $[-1,1]^3$; actually even in $(-1,1)^3$, as the above rules out zeroes on the boundary. (Alternatively, this is also not hard to see directly, observing that the topological degree of $g$ satisfies $\operatorname{deg}(g;(-1,1)^3;0) 
=\operatorname{deg}(\operatorname{id};(-1,1)^3;0)=1$ by homotopy invariance of the degree.)   
Consequently, $y$ is not injective on $\Omega_s$. 

We are in a position to conclude the proof of Theorem~\ref{thm:lav3d}. Let us define
\begin{align*}
	\mathcal{A} & := \{ \sigma \in A_{N} \cap B_{N} | \, \text{\eqref{e:claim-eps-1} holds}\},\\
	\mathcal{B} & := \{ \sigma \in A_{N} \cap B_{N} | \, \text{\eqref{e:claim-eps-2} holds}\} \setminus \mathcal{A}.
\end{align*}
Since one of inequalities \eqref{e:claim-eps-1} or \eqref{e:claim-eps-2} holds for every $\sigma \in A_{N} \cap B_{N}$,
we have that $\mathcal{A} \cup \mathcal{B} = A_{N} \cap B_{N}$. 
While by construction we clearly have that $\mathcal{A} \cap \mathcal{B} = \varnothing$. Arguing as in~\eqref{e:en-estimate}, applying Proposition~\ref{p:lower-bound} we estimate the energy~$E_{s}(y)$ as
\begin{align*}
	E_{s} (y) & \geq c \int_{\mathcal{A}} \int_{-1}^{1} \int_{-1}^{1} \big| \| \nabla y ( x_{1}, x_{2}, x_{3})\|_{2} - 1\big|^{p} dx_{1} \, dx_{2} \, dx_{3} 
	\\
	&
	\qquad +  c \int_{\mathcal{A}} \int_{-1}^{1} \int_{-1}^{1} \big| \| \nabla y (x_{1}, x_{2}, x_{3}) \|_{2} - 1\big|^{2} dx_{1} \, dx_{2} \, dx_{3} \nonumber
	\\
	&
	\qquad + c   \int_{\mathcal{B} + 4} \int_{-1}^{1} \int_{-1}^{1} \big| \| \nabla y (x_{1}, x_{2}, x_{3}) \|_{2} - 1\big|^{p} dx_{1} \, dx_{3} \, dx_{2} \nonumber
	\\
	&
	\qquad + c   \int_{\mathcal{B} + 4} \int_{-1}^{1} \int_{-1}^{1} \big| \| \nabla y (x_{1}, x_{2}, x_{3}) \|_{2} - 1\big|^{2} dx_{1} \, dx_{3} \, dx_{2} \nonumber
	\\
	&
	\geq c \int_{\mathcal{A}} \int_{-1}^{1} \int_{-1}^{1} \big| | \partial_{2} y ( x_{1}, x_{2}, x_{3})| - 1\big|^{p} dx_{1} \, dx_{2} \, dx_{3} \nonumber
	\\
	&
	\qquad +  c \int_{\mathcal{A}} \int_{-1}^{1} \int_{-1}^{1} \big| | \partial_{2} y (x_{1}, x_{2}, x_{3}) | - 1\big|^{2} dx_{1} \, dx_{2} \, dx_{3} \nonumber
	\\
	&
	\qquad + c   \int_{\mathcal{B} + 4} \int_{-1}^{1} \int_{-1}^{1} \big| | \partial_{3} y (x_{1}, x_{2}, x_{3}) | - 1\big|^{p} dx_{1} \, dx_{3} \, dx_{2} \nonumber
	\\
	&
	\qquad + c   \int_{\mathcal{B} + 4} \int_{-1}^{1} \int_{-1}^{1} \big| | \partial_{3} y (x_{1}, x_{2}, x_{3}) | - 1\big|^{2} dx_{1} \, dx_{3} \, dx_{2} \nonumber
	\\
	&
	\geq 2^{p-1}c \int_{\mathcal{A}} \int_{-1}^{1} \Bigg| \int_{-1}^{1} ( | \partial_{2} y ( x_{1}, x_{2}, x_{3})| - 1)dx_{2} \Bigg|^{p}  dx_{1} \, dx_{3} \nonumber
	\\
	&
	\qquad + 2^{p-1} c   \int_{\mathcal{B} + 4} \int_{-1}^{1} \Bigg|  \int_{-1}^{1} ( | \partial_{3} y (x_{1}, x_{2}, x_{3}) | - 1)  dx_{3} \Bigg|^{p} \, dx_{1} \, dx_{2} \nonumber
	\\
	&
	\geq 2^{p-1} c \, \varepsilon  \frac{18}{10} \, s . \nonumber
\end{align*}
All in all, we have shown that any deformation~$y \in W^{1, \infty}(\Omega_{s};\RR^{3}) \cap \cY_{s}$ satisfying~\eqref{e:en-bound-M} has energy $E_{s}(y)\geq \delta s$ for some positive constant~$\delta$ independent of~$s$. Thus,~\eqref{e:3d-step2} holds and the proof of the proposition is concluded. 
\end{proof}

\begin{rem}\label{rem:3}
	Similarly to Remark~\ref{rem:1}, we point out that the Lavrentiev phenomenon in dimension $d=3$ is valid if we replace $W^{1, \infty} (\Om_{s}; \RR^{3})$ with~$W^{1,r} (\Om_{s}; \RR^{3})$ for $r>\frac{6q}{q-2}$. As in~\eqref{neq:distprtion-int}, we would indeed have that for $y \in W^{1,r} (\Om_{s}; \RR^{2})\cap \cY_{s}$ with $E_{s} (y) <+\infty$ the distortion coefficient $K_{y} = \frac{|\nabla y|^3}{\det \nabla y}$ belongs to~$L^{\eta} (\Om_{s})$ for~$\eta:=\frac{rq}{3q+r}$, $\eta \in (2, q)$. This implies that any competitor $y \in W^{1,r} (\Om_{s}; \RR^{3})\cap \cY_{s}$ with  energy $E_{s}(y) \approx s$ still fulfills~\eqref{e:eps} and~\eqref{e:eps-2} of Lemma~\ref{p:proof3}. Hence, the proof of the lower bound of~$E_{s}(y)$ in Proposition~\ref{p:lav3d-2} proceeds as in the $W^{1,\infty}$-case.
\end{rem}

\section*{Declarations}

\subsection*{Funding} The work of S.A.\ was partially funded by the Austrian Science Fund (FWF) through the projects P35359-N and ESP-61. 
The work of S.K.\ was supported by the Czech-Austrian bilateral grants 21-06569K (GA\v{C}R-FWF) and 8J23AT008 (M\v{S}MT-WTZ mobility).
A.M.\ was supported by the European Unions Horizon 2020 research and innovation programme under the Marie Sk{\l}adowska-Curie grant agreement No 847693.

\subsection*{Authors' contributions} The authors equally contributed to the manuscript.

\subsection*{Ethics approval and consent to participate} Not applicable.

\subsection*{Competing interests} The authors declare that they have no competing interests.

\subsection*{Consent for publication} The publication has been approved by all co-authors.

\subsection*{Availability of data and materials} Not applicable.

\subsection*{Acknowledgements} Not applicable.

\bibliographystyle{plain}
\bibliography{NLEbib}

\begin{thebibliography}{10}

\bibitem{Ant05B}
S.~S. Antman.
\newblock {\em Nonlinear problems of elasticity}, volume 107 of {\em Applied
  Mathematical Sciences}.
\newblock Springer, New York, second edition, 2005.

\bibitem{Ba77a}
J.~M. Ball.
\newblock Convexity conditions and existence theorems in nonlinear elasticity.
\newblock {\em Arch. Rational Mech. Anal.}, 63(4):337--403, 1977.

\bibitem{Ba81a}
J.~M. Ball.
\newblock {Global invertibility of Sobolev functions and the interpenetration
  of matter.}
\newblock {\em {Proc. R. Soc. Edinb., Sect. A, Math.}}, 88:315--328, 1981.

\bibitem{Bal1982}
J.~M. Ball.
\newblock Discontinuous equilibrium solutions and cavitation in nonlinear
  elasticity.
\newblock {\em Philos. Trans. Roy. Soc. London Ser. A}, 306(1496):557--611,
  1982.

\bibitem{Ba02a}
J.~M. Ball.
\newblock Some open problems in elasticity.
\newblock In {\em Geometry, mechanics, and dynamics}, pages 3--59. Springer,
  New York, 2002.

\bibitem{Ball2010}
J.~M. Ball.
\newblock {\em Progress and puzzles in nonlinear elasticity}, pages 1--15.
\newblock Springer Vienna, Vienna, 2010.

\bibitem{BalMiz1985}
J.~M. Ball and V.~J. Mizel.
\newblock One-dimensional variational problems whose minimizers do not satisfy
  the {E}uler--{L}agrange equation.
\newblock {\em Arch. Rational Mech. Anal.}, 90(4):325--388, 1985.

\bibitem{BarHenMor2017}
M.~Barchiesi, D.~Henao, and C.~Mora-Corral.
\newblock Local invertibility in {S}obolev spaces with applications to nematic
  elastomers and magnetoelasticity.
\newblock {\em Arch. Ration. Mech. Anal.}, 224:743--816, 2017.

\bibitem{BarDou2020}
C.~W. Barney, C.~E. Dougan, K.~R. McLeod, and et~al.
\newblock Cavitation in soft matter.
\newblock {\em Proc. Natl. Acad. Sci. USA}, 117(17):9157--9165, 2020.

\bibitem{BouHenMol2020}
O.~Bouchala, S.~Hencl, and A.~Molchanova.
\newblock Injectivity almost everywhere for weak limits of {S}obolev
  homeomorphisms.
\newblock {\em J. Funct. Anal.}, 279(7):108658, 32, 2020.

\bibitem{ButBel1995}
G.~Buttazzo and M.~Belloni.
\newblock A survey on old and recent results about the gap phenomenon in the
  calculus of variations.
\newblock In {\em Recent developments in well-posed variational problems},
  volume 331 of {\em Math. Appl.}, pages 1--27. Kluwer Acad. Publ., Dordrecht,
  1995.

\bibitem{Cia88B}
Ph.~G. Ciarlet.
\newblock {\em Mathematical elasticity. {V}ol. {I}}, volume~20 of {\em Studies
  in Mathematics and its Applications}.
\newblock North-Holland Publishing Co., Amsterdam, 1988.
\newblock Three-dimensional elasticity.

\bibitem{CiaNe87a}
Ph.~G. Ciarlet and J.~Ne\v{c}as.
\newblock {Injectivity and self-contact in nonlinear elasticity.}
\newblock {\em {Arch. Ration. Mech. Anal.}}, 97:173--188, 1987.

\bibitem{ConDeLel2003}
S.~Conti and C.~De~Lellis.
\newblock Some remarks on the theory of elasticity for compressible
  {N}eohookean materials.
\newblock {\em Ann. Sc. Norm. Super. Pisa Cl. Sci.}, 2:521--549, 2003.

\bibitem{DolHenMol2022}
A.~Dole\v{z}alov\'a, S.~Hencl, and A.~Molchanova.
\newblock Weak limit of homeomorphisms in {$W^{1,n-1}$}: invertibility and
  lower semicontinuity of energy.
\newblock 2022.

\bibitem{FoGa95a}
I.~{Fonseca} and W.~{Gangbo}.
\newblock {Local invertibility of Sobolev functions.}
\newblock {\em {SIAM J. Math. Anal.}}, 26(2):280--304, 1995.

\bibitem{Foss2003}
M.~Foss.
\newblock Examples of the {L}avrentiev phenomenon with continuous {S}obolev
  exponent dependence.
\newblock {\em J. Convex Anal.}, 10(2):445--464, 2003.

\bibitem{FosHruMiz2003b}
M.~Foss, W.~Hrusa, and V.~J. Mizel.
\newblock The {L}avrentiev phenomenon in nonlinear elasticity.
\newblock {\em J. Elasticity}, 72(1-3):173--181, 2003.
\newblock Essays and papers dedicated to the memory of Clifford Ambrose
  Truesdell III. Vol. III.

\bibitem{FoHruMi03a}
M.~{Foss}, W.~J. {Hrusa}, and V.~J. {Mizel}.
\newblock {The Lavrentiev gap phenomenon in nonlinear elasticity.}
\newblock {\em {Arch. Ration. Mech. Anal.}}, 167(4):337--365, 2003.

\bibitem{GraKruMaiSte19a}
D.~Grandi, M.~Kru\v{z}\'{\i}k, E.~Mainini, and U.~Stefanelli.
\newblock A phase-field approach to {E}ulerian interfacial energies.
\newblock {\em Arch. Ration. Mech. Anal.}, 234(1):351--373, 2019.

\bibitem{HeaKroe09a}
T.~J. Healey and S.~Kr\"omer.
\newblock {Injective weak solutions in second-gradient nonlinear elasticity.}
\newblock {\em {ESAIM, Control Optim. Calc. Var.}}, 15(4):863--871, 2009.

\bibitem{HenMor2010}
D.~Henao and C.~Mora-Corral.
\newblock Invertibility and weak continuity of the determinant for the
  modelling of cavitation and fracture in nonlinear elasticity.
\newblock {\em Arch. Ration. Mech. Anal.}, 197:619--655, 2010.

\bibitem{HeMo15a}
D.~Henao and C.~Mora-Corral.
\newblock Regularity of inverses of sobolev deformations with finite surface
  energy.
\newblock {\em Journal of Functional Analysis}, 268(8):2356--2378, 2015.

\bibitem{HeMoOl21a}
D.~Henao, C.~Mora-Corral, and M.~Oliva.
\newblock Global invertibility of {S}obolev maps.
\newblock {\em Adv. Calc. Var.}, 14(2):207--230, 2021.

\bibitem{Hen-Str-2020}
D.~Henao and B.~Stroffolini.
\newblock Orlicz--{S}obolev nematic elastomers.
\newblock {\em Nonlinear Anal.}, 194:111513, 21, 2020.

\bibitem{HenKos2014}
S.~Hencl and P.~Koskela.
\newblock {\em Lectures on mappings of finite distortion}, volume 2096 of {\em
  Lecture Notes in Mathematics}.
\newblock Springer, Cham, 2014.

\bibitem{IwaOnn2009}
T.~Iwaniec and J.~Onninen.
\newblock Hyperelastic deformations of smallest total energy.
\newblock {\em Arch. Ration. Mech. Anal.}, 194(3):927--986, 2009.

\bibitem{Kroe20a}
S.~Kr\"{o}mer.
\newblock Global invertibility for orientation-preserving {S}obolev maps via
  invertibility on or near the boundary.
\newblock {\em Arch. Ration. Mech. Anal.}, 238(3):1113--1155, 2020.

\bibitem{La1927a}
M.~A. Lavrentieff.
\newblock {Sur quelques probl\`emes du calcul des variations.}
\newblock {\em {Ann. Mat. Pura Appl. (4)}}, 4:7--28, 1927.

\bibitem{ManVill1998}
J.~Manfredi and E.~Villamor.
\newblock An extension of {R}eshetnyak's theorem.
\newblock {\em Indiana Univ. Math. J.}, 47(3):1131--1145, 1998.

\bibitem{Man1934}
B.~Mani\`a.
\newblock Sopra una classe particolare di integrali doppi del {C}alcolo delle
  {V}ariazioni.
\newblock {\em Ann. Mat. Pura Appl.}, 13(1):91--104, 1934.

\bibitem{Ma13a}
J.~Mawhin.
\newblock Variations on {P}oincar{\'e}--{M}iranda’s theorem.
\newblock {\em Advanced Nonlinear Studies}, 13(1):209--217, 2013.

\bibitem{MolVod2020}
A.~Molchanova and S.~Vodopyanov.
\newblock Injectivity almost everywhere and mappings with finite distortion in
  nonlinear elasticity.
\newblock {\em Calc. Var. PDE}, 59:17, 2020.

\bibitem{MulSpe1995}
S.~M\"{u}ller and S.~J. Spector.
\newblock An existence theory for nonlinear elasticity that allows for
  cavitation.
\newblock {\em Arch. Ration. Mech. Anal.}, 131:1--66, 1995.

\bibitem{Pa08a}
Olivier Pantz.
\newblock The modeling of deformable bodies with frictionless (self-)contacts.
\newblock {\em Arch. Ration. Mech. Anal.}, 188(2):183--212, 2008.

\bibitem{Sil97B}
M.~{\v{S}ilhav\'y}.
\newblock {\em The mechanics and thermodynamics of continuous media}.
\newblock Texts and Monographs in Physics. Springer, Berlin, 1997.

\bibitem{SwaZie2004}
D.~Swanson and W.~P. Ziemer.
\newblock The image of a weakly differentiable mapping.
\newblock {\em SIAM J. Math. Anal.}, 35(5):1099--1109, 2004.

\bibitem{Tan1988}
Q.~Tang.
\newblock Almost-everywhere injectivity in nonlinear elasticity.
\newblock {\em Proc. Roy. Soc. Edinburgh Sect. A}, 109(1-2):79--95, 1988.

\bibitem{Sve88a}
V.~\v{S}ver\'{a}k.
\newblock Regularity properties of deformations with finite energy.
\newblock {\em Arch. Rational Mech. Anal.}, 100(2):105--127, 1988.

\end{thebibliography}

\end{document}